\documentclass[11pt,leqno]{amsart}
\usepackage{amsthm,amsfonts,amssymb,amsmath,oldgerm}
\numberwithin{equation}{section}
\usepackage{fullpage}
\usepackage{color}
\usepackage{calrsfs}
\DeclareMathAlphabet{\pazocal}{OMS}{zplm}{m}{n}
\usepackage{graphicx, xcolor}
\usepackage{subcaption}

\graphicspath{{draws/}}

\def\eps{\varepsilon }

\def\eps{\varepsilon}

\newcommand\br{\begin{remark}}
\newcommand\er{\end{remark}}
\newcommand\bp{\begin{pmatrix}}
\newcommand\ep{\end{pmatrix}}
\newcommand{\be}{\begin{equation}}
\newcommand{\ee}{\end{equation}}
\newcommand\ba{\begin{equation}\begin{aligned}}
\newcommand\ea{\end{aligned}\end{equation}}

\newcommand{\bap}{\begin{app}}
\newcommand{\eap}{\end{app}}
\newcommand{\begs}{\begin{exams}}
\newcommand{\eegs}{\end{exams}}
\newcommand{\beg}{\begin{example}}
\newcommand{\eeg}{\end{exaplem}}
\newcommand{\bpr}{\begin{proposition}}
\newcommand{\epr}{\end{proposition}}
\newcommand{\bt}{\begin{theorem}}
\newcommand{\et}{\end{theorem}}
\newcommand{\bc}{\begin{corollary}}
\newcommand{\ec}{\end{corollary}}
\newcommand{\bl}{\begin{lemma}}
\newcommand{\el}{\end{lemma}}
\newcommand{\bd}{\begin{definition}}
\newcommand{\ed}{\end{definition}}
\newcommand{\brs}{\begin{remarks}}
\newcommand{\ers}{\end{remarks}}

\newcommand{\const}{\text{\rm constant}}

\newtheorem{theorem}{Theorem}[section]
\newtheorem{proposition}[theorem]{Proposition}
\newtheorem{corollary}[theorem]{Corollary}
\newtheorem{lemma}[theorem]{Lemma}

\theoremstyle{remark}
\newtheorem{remark}[theorem]{Remark}
\theoremstyle{definition}
\newtheorem{definition}[theorem]{Definition}

\newtheorem{example}[theorem]{Example}

\newcommand{\beq}{\begin{equation}}
\newcommand{\eeq}{\end{equation}}

\title{ Continuous guts poker and numerical optimization of generalized recursive games }

\author{Kevin Buck}
\address{Indiana University, Bloomington, IN 47405}
\email{kevbuck@iu.edu}
\thanks{Research of K.B. was partially supported under NSF grant no. DMS-0300487}
\author{Jae Hwan Lee}
\address{U.C. Berkeley, Berkely, CA 94720 }
\email{jaehwanlee@berkeley.edu}
\thanks{Research of J.L. was partially supported NSF grant no. DMS-2051032 (REU).}
\author{Jacob Platnick}
\address{Northwestern University, Evanston, IL 60208}
\email{jacobplatnick2024@u.northwestern.edu}
\thanks{Research of J.P. was partially supported under NSF grant no. DMS-2051032 (REU).}
\author{Aric Wheeler}
\address{Indiana University, Bloomington, IN 47405}
\email{awheele@iu.edu}
\thanks{Research of A.W. was partially supported under 
NSF grant no. DMS-0300487}
\author{Kevin Zumbrun}
\address{Indiana University, Bloomington, IN 47405}
\email{kzumbrun@iu.edu}
\thanks{Research of K.Z. was partially supported under NSF grant no. DMS-0300487}

\begin{document}

\begin{abstract}
We study a type of generalized recursive game introduced by Castronova, Chen, and Zumbrun
featuring increasing stakes, with an emphasis on continuous guts poker and $1$ v. $n$ coalitions.
Our main results are to develop practical numerical algorithms with rigorous underlying theory
for the approximation of optimal multiplayer strategies, and to use these to obtain a number
of interesting observations about guts.
Outcomes are a striking 2-strategy optimum for $n$-player coalitions,
with asymptotic advantage approximately $16\%$;
convergence of Fictitious Play to symmetric Nash equilibrium;
and a malevolent interactive $n$-player ``bot'' for demonstration.
For a mild variant of Guts known as the ``Weenie rule'' we find the surprising different result that the Nash 
equlibrium solution is in fact strong; that is, it is optimal against arbitrary coalitions.
\end{abstract}

\date{\today}
\maketitle
\tableofcontents


\section{Introduction}\label{s:intro}
In this paper, we study a class of single-state generalized recursive games introduced in \cite{CCZ}, 
typified by the popular poker game ``Guts,''
with an eye toward developing a practical set of theoretical and numerical tools for optimizing multiplayer
strategy.

\medskip

{\bf Generalized recursive games.}
{Recursive games}, introduced by Everett \cite{E}, are games like Markov chains consisting of a set
of possible states, for which the outcome of
a single round is to either terminate with a given fixed payoff or to proceed to another state, with various (fixed)
probabilities.
This is closely related to the concept of {stochastic game} introduced by Shapley \cite{Sh1},
in which players move from state to state with various probabilities, receiving payoffs at the same time without
termination. Indeed, if $\rho_0$ is the probability of termination in a one-shot recursive game,
$a_0$ is the associated payoff, and $\rho, \dots, \rho_\ell$ are the probabilities of moving to states 
$1,\dots, \ell$, then this is equivalent to a {\it variable-stakes} stochastic game with transition probabilities
$\rho_1, \dots, \rho_\ell$, one-shot payoff $\alpha:= \rho_0 a_0$, and new stakes adjusted by
multiplication factor $\beta:=(1-\rho_0)$.  That is, the recursive game of Everett may be viewed
as a variable-stakes version of the stochastic game of Shapley, where the stakes are multiplied by a factor
\be\label{betaless}
0\leq \beta\leq 1.
\ee

The class of games introduced in \cite{CCZ}, denoted here as {\it generalized recursive games} (gRG) 
have the variable-stake stochastic structure just described, but with the condition \eqref{betaless}
relaxed to
\be\label{betagtr}
0\leq \beta.
\ee
That is, in its most general form, 
in each round a game in state $i$ moves to a new state $j$ with transition probability $\rho_{ij}$, 
receiving a one-shot payoff $\alpha_{ij}$ and adjusting the stakes by factor $\beta_{ij}\geq 0$,
where $i, j=1,\dots, \ell$ enumerate the possible states, and $\rho_{ij}$, $\alpha_{ij}$, $\beta_{ij}$
are functions of the strategies chosen by the players in the game.

Here, as in \cite{CCZ}, we will restrict to the simplest case of a {\it single state}, 
so that there is a single one-shot payoff $\alpha$ and a single stakes multiplier $\beta$
that is nonnegative but not necessarily bounded by one.
We will assume further that the game is {\it zero sum} in the sense that the one-shot game described by
$\alpha$ is zero sum.
This does not necessarily mean that total payoff over the infinite course of the game is zero sum,
as increasing stakes allow the possibility that some part of the stakes remain ``indefinitely in play''
and inaccessible to all players, serving as an effective mutual loss.
For the same reason, von Neumann's maximin principle does not necessarily hold in the long run even in the 
two-player case, though it of course holds for the one-shot payoff function $\alpha$.

\medskip

{\bf Discrete and continuous guts poker.}
An archetypal example, and our particular target here is the popular poker game ``Guts'' \cite{S,W1}
which can be played with any number of players $n \geq 2$ and $2-$, $3-$, or $5$-card hands with standard
poker ordering.
Players make an initial one unit ante into a pot. Hands are dealt, and on the count of three players
either ``hold'' or ``drop'' their hands, with no further betting or cards dealt.
If only one player holds, they win the pot and the round is terminated.  If no players hold, the game is
redealt, starting over.  If $m\geq 2$ players hold, the player with highest hand wins the pot and the 
remaining $m-1$ players must match it, so that the stakes increase by factor $m-1$. A new hand
is then dealt to all players and the game played in the same way but with now higher stakes, this process
continuing until the round is terminated.
In actual poker there are many identical rounds of play; here, we view each round as an individual game.

Following \cite{CCZ}, we will treat a simplified, continuous version of guts, in which hands are replaced
by random variables $p_i$ uniformly distributed on $[0,1]$, with higher value corresponding to a higher hand.
This avoids combinatorial details coming from lack of replacement, while keeping the essential features 
of the game.
We then discretize the values on $[0,1]$ to obtain an arbitrarily-near
game suitable for numerical analysis, returning full circle to a (slightly simpler) finite game.

A ``pure'' strategy for player $i$, indexed by $p_i^*\in [0,1]$, is the threshold type strategy
to hold for $p_i> p_i^*$ and otherwise drop.
As noted in \cite{CCZ}, one might also consider more general type pure strategies to hold for $p_i$
in a specified subset $\mathcal{S}_i\subset [0,1]$; however, these are evidently majorized by the 
threshold strategies $p_i^*=|\mathcal{S}_i|$, and so may be ignored.
A ``mixed,'' or ``blended'' strategy is a random mixture of pure strategies with a given probability weight.

\medskip

{\bf Coalitions, and value of multiplayer games.}
For a symmetric multi-player game such as the one-shot (payoff $\alpha$) game associated with guts, 
there always exists a symmetric Nash Equilibrium 
\cite{N,W2}, meaning a collection of identical strategies from which deviation by a single player
does not improve their payoff hence there is no motivation to leave. The value of this symmetric
Nash equilibrium is necessarily zero, by the zero-sum assumption.

If the Nash equilibrium is {\it strong}, meaning that deviation of any subset of players does not improve
their joint payoff, then this is in fact the value of the game by any reasonable definition.
However, typically this is not the case, and indeed it is not the case for one-shot continuous guts \cite{CCZ}.
In the latter scenario, following von Neumann and Morganstern \cite{vNM},
it is standard to consider the effects of coalitions in assigning a value.
Taking the worst-case point of view, we will define the value $\underline{v}$ for player one
to be the von Neumann minimax value of the two-player game obtained by considering players $2$-$n$ as a 
single entity consisting of a coalition working together.
One might call this a ``synchronous'' coalition, in that players $2$-$n$ may form mixed strategies 
blending coordinated or ``synchronized'' configurations of pure strategies.
It is straightforward to see that this value is the maximum one forceable by player one.
By the minimax principle, it is equal to the value forceable via synchronous coalition by players $2$-$n$.

An interesting question that we do not address here is what is the minimax value forceable by
players $2$-$n$ via {\it asynchronous} coalition, that is,
$$
\min_{p_2,\dots,p_n} \max_{p_1} \psi(p_1, p_2, \dots, p_n),
$$
where $\psi$ denotes the expected payoff for blended strategies with probability distributions $p_j$,
with $j=1,\dots, n$.
This corresponds to the case that players $2$-$n$ are able to organize a strategy together, but not
allowed to communicate during the game, or even to know to which round of play they are responding at a given time.
It is clear that the asynchronous minimax is greater than or equal to the maximin, or synchronous coalition value,
and strictly less than the Nash equilibrium value, except in the case the the Nash equilibrium
is strong, in which case all three values agree. 

If the Nash equilibrium is not strong, then (see
\cite[Proposition A.2]{CCZ}) the optimum synchronous coalition strategy contains strategies 
for which players $2-n$'s individual strategies do not agree.
If this optimum is {\it strict} up to symmetry, and has no pure strategy representative,
then every asynchronous mixed strategy for players $2$-$n$ gives a larger return to player $1$,
and so we may deduce that there is a gap between the minimax values forceable by players $2$-$n$
by synchrous vs. asynchronous coalition play. 
Thus, a gap between synchronous vs. asynchronous minimax values would seem to be the generic situation; 
based on our numerical experiments, it indeed appears to be the case for continuous Guts
(by examination of optimal synchronous strategies found below).

The maximin against asynchronous play on the other hand is the same as the maximin against synchronous play.
In the above-described situation that there is a gap between synchronous and asynchronous minimax values, 
therefore, there is a gap also between minimax and maximin values for asynchronous coalition play.
In this situation, for which the minimax is strictly greater than the maximin, 
one could imagine a negotiation in which both sides agree to some intermediate return, 
even though neither may force such an outcome.

\medskip

{\bf The method of Fictitious Play.}
Having framed the one-shot problem as a standard finite two-player zero-sum game,
we now discuss practical solution of such a game.
Dantzig \cite{D} demonstrated an equivalence between such games and the class of linear programming problems.
Hence, one approach would be to translate to a linear programming problem and apply the Simplex method or
one of the various interior point methods that has been developed in recent years.
Alternatively observing that the minimax problem is a convex minimization program,
one may solve by subgradient descent or related methods.

An appealing alternative of interest in its own right, involving concepts more directly
involved with the game, is the iterative {\it method of Fictitious Play} (FP)
introduced by Brown \cite{B} as a model for learning/dynamics of actual play.
In this algorithm, the game is repeated a large number of times, with players at each step
following a strategy that is the best response to the empirical probability distribution defined
by past play of the opponent player.
As shown in a remarkable paper of Robinson \cite{R}, (FP) converges for finite two-player zero-sum games
to the associated Nash equilibrium, in this case equal to the von Neumann minimax.
For multiplayer or nonzero-sum games, it may not converge, as shown by a clever counterexample of
Shapley \cite{Sh2,Sh3} for a class of nonzero-sum 3x3 two-player games, hence, by introduction
of a virtual third player with no choice in strategy, also for a class of zero-sum 3x3x1 three-player games.

This algorithm is conveniently supported for finite two-player games in the Python package Nashpy \cite{NPy}, 
requiring only input of the payoff matrix.
We will use it as the basis of our numerical methods, yielding approximate solutions of the one-shot game
associated with a (gRG).

\medskip

{\bf Fixed-point iteration: from one-shot game to long-term value.}
To assign a value to the full, possibly infinite game, we define following \cite{CCZ} a {\it termination fee}
\be\label{t}
-t_1,
\ee
typically $\leq 0$ for player one to leave the game, and similarly for player two.
For example, in the context of Guts Poker, $-t_1=-1$ corresponds to forfeiting one's ante, while
$-t_1=0$ corresponds to the convention that players are allowed at any time to collect their share of the
pot and leave play without penalty. Here, as in \cite{CCZ}, we take always the former case \emph{$-t_1=-1$}.

Evidently, the value $V_0$ forceable by player one in zero rounds of play is thus $V_0=-t_1$.
Denoting by $A$ and $B$ the matrices associated with the one-shot payoff $\alpha$ and stakes multiplier $\beta$,
so that $\alpha(x,y)=xAy^T$, $\beta(x,y)=xBy^T$, we thus have that the value forceable by player one
in precisely one round of play is
$ V_1=Value(A + BV_0), $
where $Value$ of a matrix denotes the standard minimax value.
Continuing, we find that the value $V_n$ forceable in $n$ round of play is given inductively by
\be\label{Vn}
 V_{n+1}=Value(A + BV_n).
 \ee

 By the assumption $B_{ij}\geq 0$, a consequence of $\beta\geq 0$, the sequence $V_n$ is monotone increasing
 so long as $v_1>V_0$, hence has a well-defined limit 
 \be\label{underV}
 \underline{V}:=\lim_{n\to \infty} V_n.
 \ee
 As noted in \cite{CCZ}, $\underline{V}$ is necessarily a {\it fixed point} of the value map
 \be\label{value}
 T(V):= Value(A+BV),
 \ee
 which could be $V=+\infty$.


 For $V_1>V_0$, we define the value forceable by player one to be the minimum of $\underline{V}$ and $t_2$, 
 the amount forfeited by player two to terminate the game without play.
 If $V_1\leq V_0$ on the other hand, then the value forceable by player one is $\underline{V}=V_0=-t$, and
 player one should not enter into play.
Likewise, the infimum $\overline V$ of expected returns that can be forced by player two is a fixed point of $T$.
Note that both or either could be $\pm \infty$ in general.  When $\underline V= \overline V=V$, 
with $-t_1\leq V\leq t_2$, we say
that the game {\it has value $V$}, similarly as in the nonrecursive case.

\subsection{Main results}\label{results}\label{s:results}
We now describe briefly our main results.

\subsubsection{Analytical}\label{s:analytical}
The theoretical centerpiece of \cite{CCZ} was the following ``Termination Theorem''
giving conditions for existence of a repeated single strategy guaranteeing 
a winning or neutral outcome for player one. This was used to support rigorously all of the
conclusions of that work.

\begin{proposition}[\cite{CCZ}]\label{tthm}
	Suppose for an arbitrary single-state recursive game with termination constant $-t<0$, 
	that a certain strategy for player 1 has associated payoffs $\alpha+  \beta V$ satisfying 
	\be\label{tcond}
	\hbox{\rm $\alpha\geq \alpha_0\geq 0$, $\beta \geq 0$, and
$\alpha \geq t( \beta -1) +\eps$, for some $0<\eps<t$.}  
	\ee
Then, the expected return $V_n$ upon termination of the game at the $n$th step satisfies
	\be\label{Vnest}
	V_n\geq v_n:=\alpha_0 -  \alpha_0 (1-\eps/t)^{\max\{0,n-1\}} -t(1-\eps/t)^n;
	\ee
	that is, the strategy forces a payoff with lower bound \emph{exponentially converging
	to $\alpha_0$}.
	In particular, condition \eqref{tcond} gives sharp criteria for existence of 
	a single strategy forcing $\underline{V} >0$ or $\underline {V}\geq 0$.
\end{proposition}

Though evidently of practical use, this result has two shortcomings: First, it is designed to sharply
differentiate winning from losing games for player one, but not to determine the sharp value of $\underline{V}$,
as is clearly of great interest in applications.  Second, it concerns only repeated strategies, whereas
it could well be that a sequence of different strategies might be necessary in order to go from $V_0=-t$
to $\underline{V}-\eps$ in the iterative sequence for $\underline{V}$.

Our first new results are the following theorems simplifying and extending the previous one.

\begin{theorem}[Transition criterion]\label{ntthm} 
	For a fixed Player 1 strategy $S$, suppose that $\alpha +\beta V>V$ for $V=v_*$
	and $\alpha +\beta V\geq V$ for $V= v^*$.
	Then, for any initial value $V_0\in [v_*,v^*]$, a limiting value $\underline V\geq v^*$
	may be achieved by repeated play of that strategy.
	Moreover, this criterion is sharp.
\end{theorem}

\begin{theorem}[Convergence rate]\label{trate} 
	If, above, $\alpha + \beta v_*\geq v_* +\eps$, $\eps>0$, then, for $V_0\in [v_*,v^*]$,
	we have the sharp ``geometric series estimate''
	$ (v^*-V_n) \leq  (1-\eps)^n (v^*-v_*).  $
\end{theorem}

Theorems \ref{ntthm} and \ref{trate} recover for the choice $v_*=-t$, $v^*=\alpha_0$ 
the result of Theorem \ref{tthm}.
The choice $v_*=-t$, $v^*=\underline{V}$ on the other hand, 
where $\underline{V}$ is a fixed point of the value map $T$, with
$S$ taken to be the optimal fixed-point strategy for player 1, gives a sharp criterion for existence of a
repeated strategy forcing $\underline{V}$, remedying the first deficiency mentioned above.
Moreover, Theorem \ref{ntthm} may be applied repeatedly over a series of interlacing intervals
$[v_{*,i}, v^{*,i}]$ to reach from $-t$ to $\underline{V}$ by finite repetitions of a sequence of different strategies
$S_i$, thus remedying the second as well.

\medskip

{\it Overshoot.} Theorems \ref{ntthm}-\ref{trate} are meant to be used in combination
with numerical iteration of \eqref{Vn}.
This raises the question of {\it numerical overshoot}: is it possible that numerical error could
lead to an approximation $\tilde V_n$ of $V_n$ that is strictly larger than $\underline{V}$, with
subsequent approximations $\tilde V_m$, $m>n$ increasing to a different fixed point $V^*>\underline{V}$
that does not represent the value forceable by player one?
In this case, the a priori validation afforded by Theorem \ref{ntthm} would be useless, as the obtained
value $V^*$ would be incorrect.

To put things another way, what we would like is for $\underline{V}$ to be {\it strictly attracting from above}
in the sense that for $V>\underline{V}$ and $|V-\underline{V}|$ sufficiently small, 
\be\label{sa}
\hbox{\rm
$T(V)-V\leq -\delta |V-\underline{V}|$, \qquad for some $\delta>0.$}
\ee

The following theorem, and our final analytical result,
shows that this is generically the case, i.e., numerical overshoot is in general not a worry.

\begin{theorem}[Attraction from above]\label{overshoot} 
	Let the optimal fixed point strategies for players 1 and 2, defined as optimal strategies
	for game $A+B\underline{V}$, be \emph{unique}.
	Then, $\underline{V}$ is attracting from above.
\end{theorem}

\subsubsection{Numerical}\label{s:numerical}
The main conclusions of \cite{CCZ}, recorded in Propositions \ref{opt} and \ref{subopt} below,
were to (i) compute explicitly the symmetric Nash equlibrium for $n$-player guts, and (ii)
to confirm rigorously that this is an optimal strategy for player 1 against ``bloc'' coalitions
of players $2$-$n$, defined as coalitions in which all players play identical strategies,
but is not optimal against general, nonbloc coalitions.
These results were derived analytically, and were of qualitative (optimal vs. subobtimal) type.
Our goal here is by numerical investigation to obtain {\it quantitative} results.

As described in more detail later on, our first main numerical result was to develop a workable algorithm
for approximating the solution of generalized recursive games, consisting of numerical implementation
of the  fixed-point iteration \eqref{Vn} using the Fictitious Play routine supported in Nashpy \cite{NPy}.
Though we did not do it here, this could be used in principle together with the a posteriori verification
afforded by Theorem \ref{ntthm} to yield rigorous upper and lower bounds to any desired tolerance.

This algorithm worked extremely well in practice, allowing us to treat 3-player continuous guts with 201
mesh point discretization and 4- and 5-player guts at 21 mesh point discretization, all on a standard laptop
with little attempt at optimization.
Applied to 3-player guts it yielded the result that players 2-3 working in collaboration can win
$\approx 1.2\%$ of the ante of player 1, similar to a typical ``house edge'' in blackjack.
This improved on the lower bound of $\approx 0.4 \%$ obtained in \cite{CCZ}
by explicit construction using a blend of two well-chosen strategies.
Strikingly, the approximate optimal fixed-point strategy found here- that is, the optimal strategy
for the game $A+B\underline{V}$ at the fixed-point value- is quite close to the winning
strategy constructed in \cite{CCZ},
in particular consisting of a blend of just two strategies.
See \cite{CCZ} for some intuition behind this choice.
Moreover, this strategy yields $\alpha - \beta t_1> -t_1$, verifying by the condition of Theorem \ref{ntthm}
that {\it repeated application of this single strategy suffices to yield $\underline{V}$.}

Experiments pitting player 1 versus players $2$-$n$ with $n=4$ and $5$ yielded an optimal
strategy with a strikingly simple``pseudo-bloc'' structure, 
in the sense that players $3$-$n$ play always the same strategy
while player $2$ sometimes deviates. Imposing this structure by force and solving the resulting reduced 
strategy-set game allowed us to treat the 1 v. n problem up to $n=15$ with a fairly high discretization
of 101 mesh points. As described later on, this appears to yield a smooth curve of payoff with respect to $n$,
with asymptotic value $\approx 16\%$ of player 1's ante as $n\to \infty$.

The optimal pseudo-bloc strategy was then coded into a continuous guts-playing ``bot'' to form an instructional
interactive game, in which players test their skill against a virtual $n$-player coalition armed with the
optimal game-theoretic result. The authors and fellow REU students found this game entertaining and
even mildly addictive.

A side-experiment was to code up 3-player Fictitious Play for continuous guts with players responding as individual
agents.  One might imagine some kind of oscillatory behavior with different 
temporary coalitions forming and dissolving.  However, the result was convergence to (the symmetric Nash) equilibrium
similarly as in the 2-player case.
This gives a sense in which the Nash equilibrium is relevant to the 3-player game.


Last but not least, we considered the ``Weenie rule'' variant of Guts discussed in \cite[Appendix B]{CCZ},
extending the investigation of \cite{CCZ} to the nonbloc case.
Surprisingly, we found for this slight modification that the Nash equlibrium strategy is actually
optimal, not only agains bloc strategies, but against arbitrary player $2$-$n$ coalitions.
That is, continuous Guts with the Weenie rule is a rare example of a realistic multi-player game
with a simple and explicit exact solution!

\subsection{Discussion and open problems}\label{s:disc}
Our analysis resolves the main questions posed in \cite{CCZ}
regarding general single-state generalized recursive games with termination fee, 
and continuous guts in particular.
It might be interesting to consider also the ``gambler's ruin'' problem in which play continues until
the players resources are depleted.
Another important open problem is the analysis of the actual card-game version of Guts Poker.
As described in \cite{CCZ}, this involves some interesting combinatorial difficulties having to do with
lack of replacement of cards within the deck, and associated lack of independence between probabilities of
players hands.

The issue of asynchronous vs. synchronous coalition seems very interesting to explore further from a
philosophical and practical point of view. The development of an efficient and convenient numerical algorithm
for the evaluation of the asynchronous maximin $\max_x \min_{y,z}\psi(x,y,z)$ we view as an important
open problem even for standard one-shot games.
In particular, it is a very interesting question whether there exists a winning {\it asynchronous}
player $2$-$n$ coalition strategy for standard continuous guts, in the sense that it forces a 
negative expected return for player 1.

Likewise, the observed convergence to Nash equilibrium of Fictitions Play for 3-player guts seems
quite interesting from the point of view of behavior in multi-player games.
An extremely interesting open problem would be to {\it prove} convergence of (FP) for 3- and or n-player guts,
and identify the property(ies) that
guarantee it, particulary if these might extend to other relevant problems.

Related to the last two problems is to investigate Fictitious Play for the modified game
in which players 2-n pool their winnings. It is straightforward to see that a symmetric Nash equilibrium
for the original game is also a Nash equilibrium for the modified one.  But, it may be that other
interesting new equilibria exist corresponding to an asynchronous player 2-n coalition.
For continuous 3-player guts, (FP) for this modified game is seen again to converge to the symmetric equilibrium;
it is an interesting open question whether there exist other equilibria favoring players 2-3.

More generally, one may consider the question for general $n$-player symmetric zero-sum games whether it is
possible for Fictitious Play to select a winning coalition strategy, or, say, oscillate between different such
approximate coalitions. We examine this and other issues to do with Fictitious Play in Appendix \ref{s:convFP}.

Finally, it may be interesting to consider the case of $m$ v. $n$ coalitions for general $m$, $n$.
For example, a study of the $2$ v. $2$ case returned a value of $\underline{V}=\overline{V}=0$,
showing that this game is ``fair.'' The optimal strategies, though, were different from the
symmetric Nash equilibrium for the original $4$-player game,
being of a ``blended, non-bloc'' type similar to that of the optimal strategy
for players $2$-$3$ in the $1$ v. $2$ game; see Section \ref{s:numres}.
There is no inherent obstacle to the application of our numerical scheme to this more general
case, the only additional work being to code the associated payoff matrices.

\medskip
{\bf Acknowledgement:} This work was carried out with the aid of opensource packages
Desmos, Nashpy, and SciPy. J.L. and J.P. thank Indiana University, especially
REU director Dylan Thurston and administrative coordinator Mandie McCarty, for their
hospitality during the REU program in which this work was carried out.
We also thank the UITS system at Indiana University for the use of supercomputer cluster Carbonate.
All code used in the investigations of this project, 
as well as the interactive ``bot'' program, is publicly available 
and may be found at \cite{G}.


\section{General analysis}\label{s:gen}
We begin by establishing the general results stated in Section \ref{s:analytical}.

\subsection{A posteriori estimates}\label{s:apost}

\begin{proof}[Proof of Theorem \ref{ntthm}]
	It is sufficient to establish the result for the
	the subgame in which player one's strategy is always $S$.
	By linearity with respect to $V$ of the payoff function $\alpha + \beta V$, together
	with the endpoint assumptions $\alpha + (\beta-1) v_* >0$ and
	$\alpha + (\beta-1) v_* \geq 0$, we have
	$\alpha + (\beta-1) V>0$ for any $V\in [v_*,v^*)$, hence there
	is no fixed point of the value map $T$ (for the reduced game) in this half-open interval.
	In particular, the value $\underline{V}\geq v_*$ forced by strategy $S$, since it is a fixed point,
	must satisfy $\underline {V}\geq v^*$.
\end{proof}

\begin{proof}[Proof of Theorem \ref{trate}]
	Considering the subgame in which player one's strategy is always $S$,
	and defining $\theta_n:= (\underline{V}-v_n)/(\underline{V}-v_0)$, 
	it is equivalent to show that $\theta_n\leq (1-\eps)^n$.
	Rearranging, we have
	$$
	v_n= (1-\theta_n)\underline{V}+ \theta_n v_0,
	$$
	whence, by linearity of the payoff matrix, together with the assumptions, 
	$$
	T(v_n)-v_n \geq \theta_n (T(v_0)-v_0)) \geq \theta_n\eps(\underline{V}-v_0)=
	\eps (\underline{V}-v_n).
	$$
	Here, we have used in an important way superlinearity of $\min_j$.

	Thus, $v_{n+1}=T(v_n)\geq  v_n+  \eps(\underline{V}-v_n)$, and so 
	$$
	\underline{V}-v_{n+1}\leq \underline {V}- (v_n+  \eps(\underline{V}-v_n))
	=(1-\eps)(\underline{V}-v_n),
	$$
	giving by induction the asserted rate
	$ (v^*-V_n) \leq  (1-\eps)^n (v^*-v_*).  $
\end{proof}

\br\label{subsuperrmk}
Note that sublinearity of $\max_i$ wrecks the above argument for the general case in which player
one's strategy is not necessarily fixed. Indeed, it appears difficult to quantify the convergence rate
in the general case, without specifying a particular strategy sequence beforehand.
Luckily, for 3-player Guts, the optimum strategies are repeated and so we may readily estimate the convergence
rate using Theorem \ref{trate}.
\er

\subsection{Overshoot}\label{s:exams}
We next investigate the possibilities for multiple fixed points, and associated phenomena
of overshoot, and duality gap, defined as
strict inequality between $\underline{V}$ and $\overline{V}$.
We start with the simple example of an $m\times 1$ game, in which player 2 has no choice in strategy,
and only player 1's choices affect the outcome.
It is not difficult to see that, among all lines $\alpha_i + \beta_i V$, 
the maximum over $i$ can intersect the fixed-point line $F(V)\equiv V$ at at most two points,
a point $V_*$ with $\beta_i< 1$ and a point $V^*$ with $\beta_i\geq 1$, the first less than the second.
Ignoring the degenerate case $\beta_i=1$, fixed
points with $\beta_i<1$ are attracting from both sides and with $\beta_i>1$ are repelling from both 
sides so long as they remain separated.  
It follows that for $V_*<V^*$ only the point $V_*$ is a candidate for the value $\underline{V}$ of the game,
and it is always attracting from above.
Thus, there is no overshoot.  However, for player 2 initial values $>V^*$ and player 1 initial values $<V^*$,
there can be a duality gap, with $\underline{V}=V_*$ and $\overline{V}=+\infty$.

In the degenerate case $V_*=V^*$ that they intersect, 
the fixed point is attracting from below and repelling from above, hence overshoot to $+\infty$
is possible, even though the correct value of the game is $\underline{V}=V_*$ for initial
values $v_0<V_*$.
An example is the $2\times 1$ game $(\alpha_1,\alpha_2)= (0,0)$, $(\beta_1, \beta_2)=(1/2,2)$, for which
any strategy is a fixed point of $ \alpha +\beta V$ for $V=0$. Moreover, for $V<0$, we see that
$Value(\alpha+ \beta V$ is $<0$ as well, so that $\underline{V}=0$.
This fixed point is attracting from below, but it is repelling from above, so that in this case overshoot
{\it can happen}.  That is, $T(V):=Value(\alpha+ \beta V)>V$ for $V<0$ {\it and} for $V>0$, with $T(V)=V$
only for $V=\underline{V}$.
In the degenerate case that $\beta_i=1$ at the fixed point, $\alpha_i=0$ and there
is a line of fixed points, all neutral, hence both overshoot and duality gap are possible.

Combining these observations, we may now treat the general $m\times n$ case, completing our analysis.

\begin{proof}[Proof of Theorem \ref{overshoot}]
	We have only to note that, when the optimal fixed-point strategies $S_1$ and $S_2$ for players 1 and 2
	are unique, then the reduced $m\times 1$ game obtained by fixing player 2's strategy always as $S_1$
	must still have a unique optimal strategy, so that the fixed-point $\underline{V}$ must coincide
	with precisely one of the fixed points $V_*$ and $V^*$ described in the discussion above.
	Thus, $\underline{V}$ is attracting from above for the reduced $m\times 1$ game, and therefore
	(since player 2 can always do at least as well in the original game) for the full, unrestricted game.
\end{proof}

{\bf Further questions.} We have seen that $m\times 1$ games give rise to $1-2$ finite fixed points,
in a specific attracting/repelling configuration.
It seems a very interesting question how many fixed points are possible for
a general $m\times n$ game, and what configurations of attraction/repulsion.

\section{Payoff functions for continuous Guts}\label{s:payoff}
We next recall from \cite{CCZ} the derivation of the payoff functions for continuous guts.
In computing the expected immediate return $\alpha$,
we take the point of view that the ante is to be paid upon termination of the game.
Thus, for example, if two players hold, then the immediate return to the winning player is the value
$+n$ of the entire pot, and to the losing player $-n$, with the stakes at next round remaining at value $1$,
the multiplier of the pot. The immediate return to any players that drop in this scenario is $0$.

When three players hold, on the other hand,
the pot doubles, effectively paying all players one unit of additional ante in the resulting
higher-stakes game, which they will in fact never have to pay.  So the immediate returns of all players are
incremented by one unit and the stakes- and ante- are changed to 2, exactly balancing out.  Thus, the winning player
receives immediate return $n+1$ and the losing (holding) players receive return $-n+1=-(n-1)$.  Any dropping players
receive $+1$.
One may check that the immediate return is in every event
zero-sum:  if $r$ players hold, then the stakes are multiplied by $(r-1)$, giving all players an additional
``virtual ante'' of $(r-2)$.  Meanwhile, the single winning player receives $+n$ return while the $(r-1)$ 
losing (holding) players receive $-n$, for a total immediate return of $n(r-2) + n- (r-1)n=0$.
It follows by summation across events that the expected immediate return $\alpha$ is zero-sum as well.

For reference, the explicit results for $n=2$ and $n=3$ derived in \cite{CCZ} are as follows. 

\begin{proposition}[\cite{CCZ}]\label{2payprop}
For continuous 2-player Guts, 
	\ba\label{2alpha}
	\alpha(p_1^*,p_2^*)&=
	\begin{cases}
		(1-2p_1^*)(p_1^*-p_2^*) & p_2^*\leq p_1^*,\\
		(1-2p_2^*)(p_1^*-p_2^*) & p_2^*> p_1^*,
	\end{cases}\\
	\beta(p_1^*, p_2^*)&= p_1^*p_2^*+ (1-p_1^*)(1-p_2^*).
\ea
\end{proposition}

\begin{proposition}[\cite{CCZ}]\label{3prop}
For $3$-player continuous Guts, 
	\ba\label{3alpha}	
	\alpha(p_1^*,p_2^*,p_3^*)&=
		\begin{cases}
		2p_1^*-p_2^*-p_3^*+(p_3^*)^3+3(p_2^*)^2p_3^*-4p_1^*p_2^*p_3^*, &
		p_1^*<p_2^*<p_3^*,\\
		2p_1^*-p_3^*-p_2^*+(p_2^*)^3+3(p_3^*)^2p_2^*-4p_1^*p_2^*p_3^*, &
		p_1^*<p_3^*<p_2^*,\\
		2p_1^*-p_2^*-p_3^*+(p_3^*)^3-3(p_1^*)^2p_3^*+2p_1^*p_2^*p_3^*, &
		p_2^*<p_1^*<p_3^*\\
		2p_1^*-p_2^*-p_3^*+(p_2^*)^3-3(p_1^*)^2p_2^*+2p_1^*p_2^*p_3^*, &
		p_3^*<p_1^*<p_2^*,\\
		2p_1^*-p_2^*-p_3^*-2(p_1^*)^3+2p_1^*p_2^*p_3^*, &
		p_2^*<p_3^*<p_1^*,\\
		2p_1^*-p_2^*-p_3^*-2(p_1^*)^3+2p_1^*p_2^*p_3^*, &
		p_3^*<p_2^*<p_1^*,
	\end{cases}\\
	\beta&= 2-p_1^*-p_2^*-p_3^*+2p_1^*p_2^*p_3^* .
	\ea
\end{proposition}

\subsection{Bloc strategies and symmetric Nash equlibria}
Following \cite{CCZ}, we define a ``bloc'' strategy for the $n$-player game to be a strategy
$(p_1^*, p_2^*,\dots, p_2^*)$ in which players $2$-$n$ pursue identical strategies.
Restricting to bloc strategies reduces the $n$-player game to a $2$-player game that is 
a simplified version of player $1$ vs. a player $2$-$n$ coalition.
It is not difficult to see \cite{CCZ} that for a symmetric one-shot game, the optimal bloc strategy
corresponds to the symmetric Nash equilibrium guaranteed by \cite{N}.
Using the explicit payoff functions of Propositions \ref{2alpha}-\ref{3alpha} for $n=2,3$, 
along with partial information obtained in \cite{CCZ} on the bloc payoff function for general $n$,
we have the following explicit description of the symmetric Nash equilibrium, and optimality for the bloc game.

\begin{proposition}[Cor 8.9]{CCZ}]\label{opt}
For $n$-player continuous Guts, the unique symmetric Nash equilibrium is $(p_1^*, \dots, p_1^*)$, with 
	\be\label{NE}
	p_1^*=1/2^{1/(n-1)}.
	\ee
Moreover, this choice forces total value $\underline{V}=0$ for player 1 in the bloc $1$ v. $(n-1)$ version of Guts.
\end{proposition}

\br\label{Jacobrmk}
As in \cite[Prop. A.2]{CCZ}, we note that
no coalition employing a bloc strategy (even a mixed strategy) of n players 
can beat a single player in any zero-sum symmetric game, as player 1 can always switch to the same strategy
to give payoff zero.
\er

\subsection{Suboptimality in the non-bloc game}
Neither the general payoff function nor the bloc version have been computed analytically for $n$-player guts
$n\geq 4$, here or in \cite{CCZ}.
However, strategically chosen test strategies were used in \cite{CCZ} to obtain the following negative conclusion.

\begin{proposition}[Cor. 8.16 \cite{CCZ}]\label{subopt}
	For the general (i.e., nonbloc) $1$ v. $(n-1)$ version of continuous Guts, there exists
	a strategy for players $2$-$n$ forcing a negative return for player 1. In particular, the
	Nash equilibrium strategy $p_1^*$ is suboptimal for player 1 vs. an $n-1$ player coalition.
\end{proposition}

The strategy constructed in Proposition \ref{subopt} was a blend of two pure strategies: one of
bloc type $(p_2^*, \dots, p_2^*)$, and the other of ``pseudo-bloc'' type $(p_2^*, p_3^*, \dots, p_3^*)$.
As described in Section \ref{s:results}, this leaves open the question of the optimal strategy in the full
(nonbloc) $n$-player game, $n\geq 3$, along with the values $\underline{V}$ and $\overline{V}$ forceable
by players $1$ and $2$-$3$, respectively,
the answers to which are the main focus of our numerical investigations.

\section{Numerical algorithms}\label{s:numalg}
We now describe the numerical algorithms used to perform our computations.

\subsection{Numerical payoff functions}\label{npayoff}

\subsubsection{General $\alpha$-function}

In order to study general 
$n$-player Guts, we developed a function in Python that can compute the expected payoff $\alpha(p_1, p_2, ..., p_n)$ for any $n$. The algorithm is as follows. \newline

1) The strategies are first reordered from $(p_1, p_2, ..., p_n)$ to $(p_1^*,p_2^*, ..., p_n^*)$ such that $p_i^* \leq p_j^*$ for $i \leq j$; essentially, the strategies are sorted in increasing order. An array of $n$ zeroes named $value$ is initialized, with each index $i$ representing the expected payoff of the player playing the strategy $p_i^*$.\newline

2) The number of players who drop is first determined and denoted as $0 \leq d \leq n - 1$. We do not need to account for the case where all players drop, since in that scenario the expected payoff is 0 for all players. The number of players who hold then is simply $n - d$ and is denoted as $h$.\newline

3) A subset $D$ of the players of size $d$ is then chosen from $n \choose d$ possible subsets to represent the players who drop. The complement $H$ of size $h$ is then the subset of players who hold.\newline

4) Of the players in $H$, there are two possibilities for each player. Let $p^H$ be the highest strategy in $H$. Then either a player receives a hand that is at least $p^H$ in which case the player is in "fair play", or the player receives a hand that is less than $p^H$ in which case the player is guaranteed to lose. A certain number of losers $0 \leq l < h$ is and a subset $L$ that represents the losers from $h \choose l$ possibilities are chosen. This also determines the number of players in fair play $f$ and a subset $F$ that represents these players so that $F \cup L = H$.\newline

5) Now that all players have either dropped, lost, or contended in fair play, we can calculate the probability of this scenario with 
\begin{equation*}
    probability = {\displaystyle \prod_{p \in D} p \prod_{p \in F} (1 - p^H) \prod_{p \in L} (p^H - p)}
\end{equation*}
We will then multiply this probability with the expected return for each player and add to the $value$ array accordingly. Players who have dropped will have a return of $-1 + (h - 1) = h - 2$. Players who have lost will have a return of $-1 - n + (h - 1) = h - n - 2$. Players who contended in fair play will have an expected return of $(-1 + n + h - 1) \cdot \frac{1}{f} + (-1 - n + h - 1) \cdot \frac{f - 1}{f}$. \newline

6) Repeat this process with each possible number of players who drop and each possible subset of players who drop. Then, restore the original order of the player strategies given in the $value$ array.\newline

This algorithm has a runtime of $\Theta(n3^n)$ assuming constant runtime for basic operations.

\subsubsection{General $\beta$-function}
We also developed a function that can compute $\beta(p_1, ..., p_n)$ for any $n$. The algorithm for this function is as follows.\newline

1) The number of players who drop $d$ is determined. The number of players who hold is $h$. Initialize $\beta = 0$.\newline 

2) A subset $D$ of the players of size $d$ is then chosen from $n\choose d$ possible subsets to represent the players who drop. The complement $H$ of size $h$ is then the subset of players who hold.\newline

3) The probability of this scenario is then 
\begin{align*}
    probability = {\displaystyle \prod_{p \in D} p \prod_{p \in H} (1 - p)}
\end{align*}. Update $\beta = \beta + probability \cdot (h - 1)$ unless $h = 0$, in which case, $\beta = \beta + probability \cdot 1$.\newline 

4) Repeat this process with each possible number of players who drop and each possible subset of players who drop. \newline

This algorithm has a runtime of $\Theta(n2^n)$.

\subsubsection{Payoff Matrices for $n$-player Guts}

To make player 1's payoff matrix for $n$-player Guts game where player 1 is against a coalition of players 2-$n$, we used a special indexing scheme to represent the strategies of the coalition. We created an $N \times N^{n - 1}$ matrix where $N$ is our discretization of Guts and the column $j_1 + j_2 \cdot N + j_3 \cdot N^2 + ... j_{n-1} \cdot N^{n-2}$ will represent the scenario where player $i$ in the coalition plays the strategy $j_i$. The strategy $i$ played by player $p_i$ of the coalition can be retrieved from the column $x$ by taking $\lfloor (x \: mod \: N^i) / N^{i-1} \rfloor$. 
\subsubsection{Pseudo-Bloc Modifications}

In the pseudo-bloc variation of Guts in which we are concerned with only the payoff of player 1, the $\alpha$ function algorithm can be simplified as players 3-$n$ play the same strategies. In a normal player 1 vs players 2-$n$ Guts game, we have to account for all possible combinations of players dropping and holding. However, in pseudo-bloc Guts, only the number of players dropping and holding in the bloc of players 3-$n$ is relevant as all of them are playing the same strategies. This allows us to reduce the runtime complexity of the $\alpha$ function to just ${O}(n^2)$.

\subsection{Numerical fixed point iteration}\label{nit}
	To determine the value of the generalized recursive game, we fill each player's matrix $\alpha +\beta V_0$ for each players $\alpha$, $\beta$, and $V_0$. $V_0$ is some anticipated value of the game. A logical starting point is 
	$-1$ for a single player and $-N$ for a coalition of $N$ players in the $1$ v. $N$ coalition case. This is a ``worst case'' that we know a player can guarantee, since any player can play a strategy of 1 (always fold) to guarantee this value. We can instead start each player at $V_0=0$ to guarantee the game is 0-sum (and thus fictitious play will converge in the $1$ v. $N$ setting).\\
	 
	After generating this first matrix, we can run fictitious play on this matrix to determine each player's value $V_1$. We can repeat this process using each $V_n$ to compute $V_{n+1}$, and repeat this process until the player;s $V_n$ converge to some $V^*$. This is a fixed point: a value where $V^*=Value(\alpha +\beta V^*)$. If this fixed point is unique, it is the value of the game. This value iteration will be monotonic in the zero-sum setting, so it will either converge to a fixed point or ``overshoot'' and diverge to infinity. As we have shown earlier, generically
	overshoot does not occur.
	
\section{Numerical results}\label{s:numres}
Finally, we describe our numerical results.
In the $1$ v. $2$ and $1$ v. $3$ player games,
we determined the optimal strategy for the coalition, which guarantees advantage against player 1. 
This strategy has a joint value of approximately 0.013 for players 2 and 3, 
and -0.013 for player 1. 
Here, players 2 and 3 play a bloc strategy of 0.68 around 86$\%$ of the time.
In the remaining 14$\%$ of the time, 
player 2 plays a strategy of 0 while player 3 plays a strategy of 0.86. 
Player 1's best response to this mixed strategy, a strategy of 0.64, loses to both. These results were created with a discretization of 101, meaning that there are 101 possible strategies from 0 to 1. 
The payoffs of the component strategies, denoted A and B, are depicted graphically in Figure \ref{figAB}
and that of the corresponding blended strategy in Figure \ref{negativefig}, below.

  \begin{figure}
		  \includegraphics[scale=0.55]{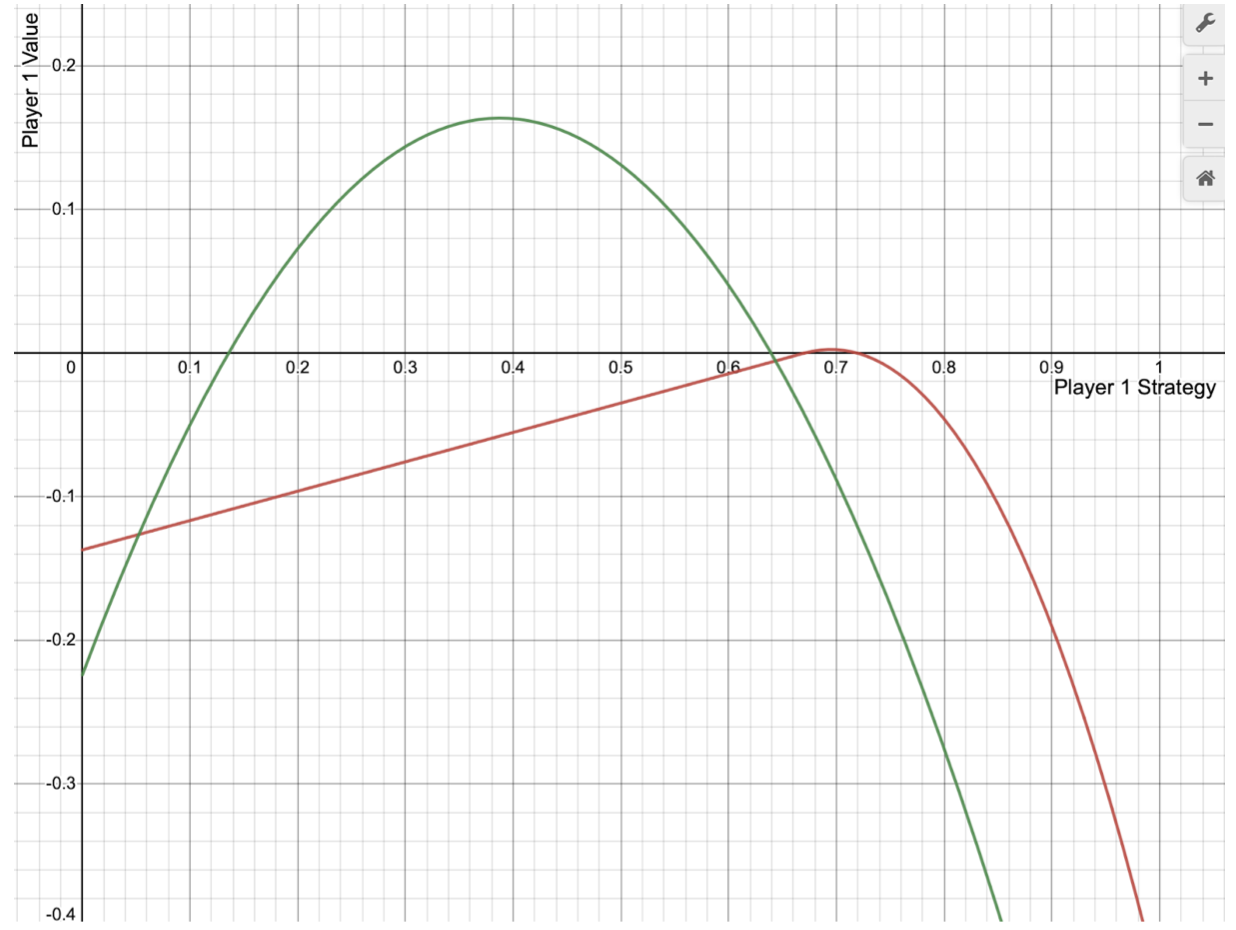}
\caption{Total expected return vs. $p_1^*$ for player 2-3 strategies A and B. }
	\label{figAB}
  \end{figure}

  \begin{figure}
		  \includegraphics[scale=0.35]{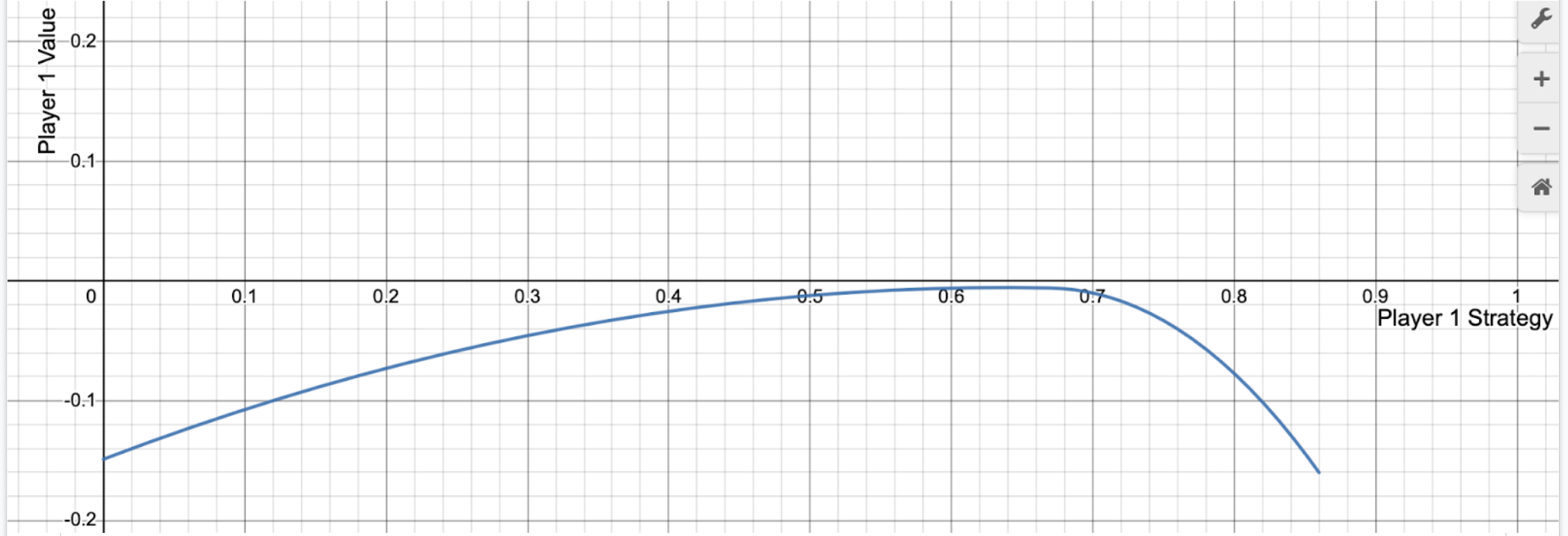}
\caption{Total expected return vs. $p_1^*$ for blended player 2-3 strategy A/B. }
	\label{negativefig}
  \end{figure}

We next computed the optimal strategies for player 1 vs. a coalition of $N$ players, with $N=$3,4, and 5.
The discretizations for these experiments was $M=21$.
	In spite of the low resolution in the discretization of strategies, a trend emerged where the coalition would randomly select either a ``bloc'' strategy where every member of the coalition would adopt the same strategy or 
a ``pseudo-bloc'' strategy where a single member of the coalition, without loss of generality player 2,
would always hold and the rest of the coalition would adopt a conservative bloc strategy with each holding
at the same threshold.
With this mixed strategy for the coalition, player one cannot force a positive return and the best they can do is to play more conservatively (than for the Nash equilibrium strategy of $p_*=1/2^{1/n}$ \cite{CCZ}) 
in order to minimize losses.
The reason for the low discretization is that treating each member of the coalition independently costs a significant amount of computational power with asymptotic run time bounded by $O(M^N I+M^N N 3^N)$, 
where $I$ denotes the number of iterations in one round of Fictitious Play, with physical run times ranging from 
minutes for $N=2$ to in excess of an hour for $N=5$.
The main reason for this is that modeling each of the $N$ player's strategy combinations is quite
computationally intensive, requiring a matrix of size $M^{N+1}$ for a discretization of $M$.

Motivated by the observed trend that the optimal coalition appears in practice to be of pseudo-bloc type,
we reduced complexity by substituting for the full game a reduced game in which the coalition could only use bloc or pseudo-bloc strategies. 
Computationally, this is significantly cheaper as now the run time is bounded by $O(M^3 I+M^3N^2)$. 
With these computational savings, we were able to explore the strategies for both player one and the coalition of $N$ players for $N$ between 2 and 15 with a discretization of $M=101$. We verified that the restrictions we placed on the coalition's strategies in the more optimized algorithm correctly reproduced the results from the independent actor method for $N=2,3,4,5$ as a crucial test. Using SciPy, we fitted a rational function to the values forceable by the coalition for $N$ between 2 and 15. 
The resulting best fit was
$$
\underline {V}(n) = 0.163 - 0.84/(n - 3.6)
$$
with an excellent $r$ squared value of  $\approx 0.9991$, and   
a predicted asymptote as $N\to\infty$ of ~0.16.
Compared to blackjack where the house edge is ~0.02; one sees that the coalition in continuous Guts has a much higher expected gain.

	Results are recorded in Table \ref{tab:table} showing, for each coalition size $N$, the value of the game from the coalitions, Player 1's best response against the optimal strategy, the coalition's bloc strategy, used a majority of the time, and the strategy played by the pseudo-bloc while player 2 plays a strategy of 0. This pseudo-bloc strategy is played around 10-30 percent of the time, increasing with the player count, however these results were noisy and thus not included. A graphical display, together with the optimized curve fit, is given in Figure \ref{fig_curvefit} below.

	\bigskip
	\begin{table}[h]
	\caption{Numerical results for the pseudo-bloc case} 
	\begin{tabular}{l | l | l | l | l}
		$N$ & Opponent Value & Player 1 Strategy & Bloc Strategy & Pseudo-bloc Strategy\\
		\hline
		1   & 0.0            & 0.5               & 0.5           & -\\
		2   & 0.0132         & 0.64              & 0.68          & 0.86\\
		3   & 0.0339         & 0.72              & 0.76          & 0.89\\
		4   & 0.0516         & 0.77              & 0.81          & 0.91\\
		5   & 0.0654         & 0.81              & 0.84          & 0.93\\
		6   & 0.0753         & 0.84              & 0.87          & 0.94\\
		7   & 0.0847         & 0.86              & 0.88          & 0.94\\
		8   & 0.0909         & 0.87              & 0.89          & 0.95\\
		9   & 0.0954         & 0.89              & 0.91          & 0.95\\
		10  & 0.1007         & 0.89              & 0.91          & 0.96\\
		11  & 0.1066         & 0.91              & 0.92          & 0.96\\
		12  & 0.1074         & 0.92              & 0.93          & 0.97\\
		13  & 0.1110         & 0.92              & 0.93          & 0.97\\
		14  & 0.1154         & 0.92              & 0.94          & 0.97\\
		15  & 0.1184         & 0.93              & 0.94          & 0.97
	\end{tabular}
	\label{tab:table}	
	\end{table}
	\medskip

  \begin{figure}
		  \includegraphics[scale=0.65]{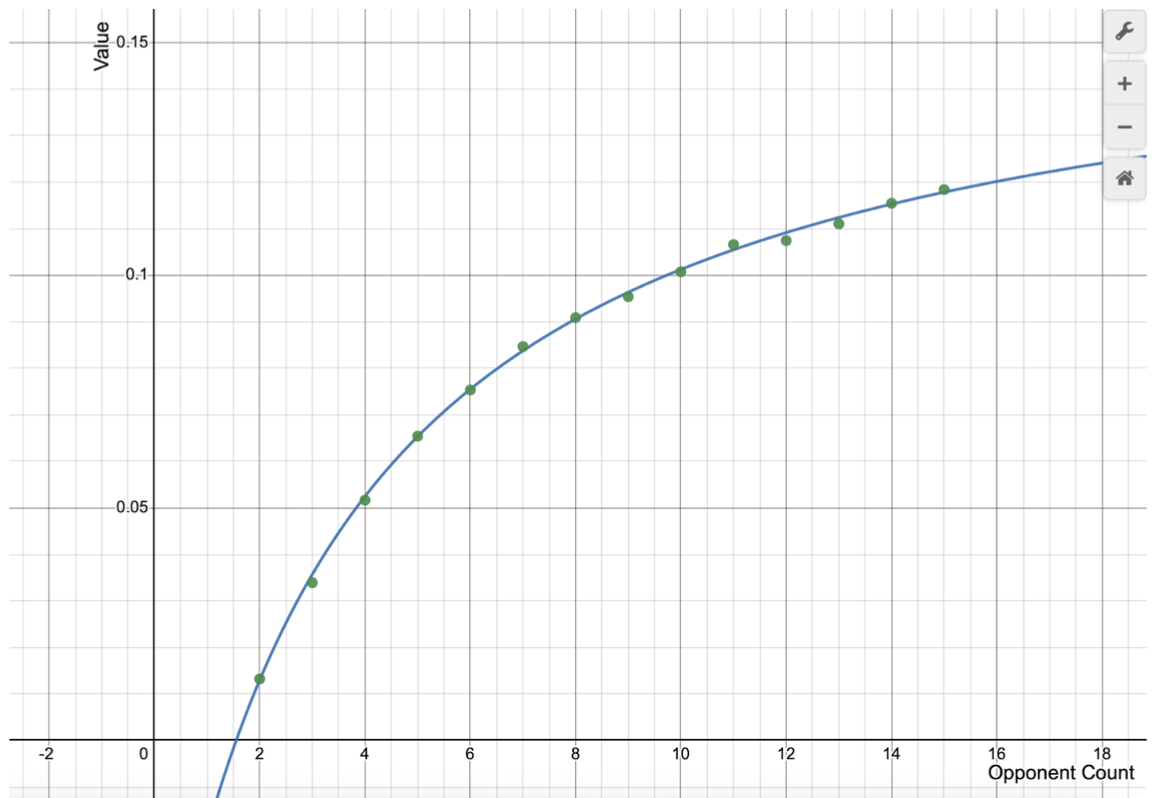}
\caption{Curve fit by rational function of total return for pseudo-bloc strategy, $1$ v. $n$. }
	\label{fig_curvefit}
  \end{figure}

 An interesting variation was the $2$ v. $2$ case, which, due to high computational complexity, 
 was carried out with relatively low discretization $M=21$.
 For this experiment, the optimal values turned out to be $\underline{V}=\overline{V}=0$, so 
 that the game has a value of $V=0$ in the strong minimax sense of von Neumann \cite{vN}.
 The optimal strategies for the two coalitions were identical and somewhat similar to the optimal
 strategies for players $2$-$3$ in the $1$ v. $2$ game.
 Namely, players pursue a bloc strategy with thresholds of $.75$-$.8$ with probability 
 $ \approx 0.95$, and non-bloc strategy with thresholds $(0, 0.95)$ with probability $0.05$.

 \subsection{The marauding guts-bot and performance}
 The optimal pseudo-bloc strategies found for a player $2$-$(n+1)$ coalition were coded into an interactive game 
 with the user playing the role of player $1$ in a $1$ v. $n$ game of continuous guts, discretized at
 101 mesh points.
 The game was tested by authors and some fellow REU students. High scores were 16 (Jay) and 26 (Jacob) for a round starting with 1 unit ante.
 Interesting informal conjectures were that even though the bot is in the long run unbeatable, certain aggressive
 or erratic patterns of play seems to increase the chances of a high score in the short term, perhaps by
 increasing variance in the outcome.
 This, and other real-world experiments on actual play and play-patterns would 
 be a very interesting direction for further investigation.

 \subsection{Three-player Fictitious Play}
	Finally, we experimented with various aspects of Fictitious Play. We saw that for 3 player free-for-all guts, the Fictitious Play algorithm converged to the Nash equilibrium. Curiously, the Fictitious Play algorithm did \emph{not} converge to the coalition strategy. The reason for this is that in the coalition, the bloc strategy does not beat the equilibrium strategy and the pseudo-bloc strategy does not allow for both players two and three to benefit while simultaneously causing player one to lose. In fact, \emph{no} strategy allows player two to benefit, player three to be at least neutral and causes player one to lose.

	More precisely, in the $1$ v. $2$ case, the 2 coalition players are incentivized to break and go to the equilibrium strategy. When one player is playing a strategy of 0, the other coalition member is incentivized to play a strategy close to the 3-player equilibrium strategy of $\frac{1}{\sqrt{2}}$, which will then also trigger the other coalition member to switch to a similar strategy. These players will be incentivized to play a strategy approaching the 3-player equilibrium strategy of $\frac{1}{\sqrt{2}}$. Through numerical experiments, we determine that no bloc strategy gains advantage against player 1 playing $\frac{1}{\sqrt{2}}$, if player 1 and player 2 are playing $\frac{1}{\sqrt{2}}$, no strategy for player 3 gains advantage, and there is no pair of strategies $S_2$, $S_3$, where if player 1 plays $\frac{1}{\sqrt{2}}$, player 2 plays $S_2$, and player 3 plays $S_3$, player 1 loses while both players 2 and 3 are at least neutral. 
	
	If two non-communicating players attempt to team up against player 1, they will be unable to meaningfully gain advantage, as even if both players play for the team and do not betray each other, players 2 and 3 each playing $\frac{1}{\sqrt{2}}$ is the only pure equilibrium, so coalitions are unlikely to arrive naturally in 3 player play. This equilibrium of $\frac{1}{\sqrt{2}}$ was observed in 3 player Fictitious Play, even when players 2 and 3 played for the team and not themselves.
	

\appendix

\section{Computational complexity}\label{s:complexity}
When experimenting with the $1$ v. $n$ case, we found that the most time-consuming portion of the calculation was the construction of the $A$ matrix.  However, we also found that in the cases we could run, 
the equilibrium was always a ``pseudo-bloc'' case, as previously described.  
In order to save computation, we thus restricted the coalition's allowed strategies to this pseudo-bloc type, 
in which players $3$-$(n+1)$ were required to play a common strategy at all times.
This greatly reduced the time necessary to calculate $A$.  Not only did this require populating a much smaller matrix (effectively populating an array of one less dimension), but the defining function for the matrix was also 
reduced from $n 3^n$ to $n^2$ cases.  This allowed us to efficiently run up to coalition size $n=15$. 
To give an idea of the run time, working on a Mac powerbook, with discretization of 101 mesh points,
it took approximately 1 minute (real time) to compute the $1$ v. $2$ solution, and approximately 
one hour (real time) to compute the $1$ v. $15$ solution under pseudo-bloc play.

In order to further improve run time, the most efficient means would be to parallelize the definition of the $A$ matrix.  If this were parallelized and run on a GPU, where many cores allow for very fast processing of parallelized operations, the runtime would be cut by a large factor.  Rewriting the program in another language could also improve runtime.  The results presented were achieved using python for ease of use.  If Fortran or C++ were implemented the results would likely be improved.
In short, the applications carried out here did not press much the limits of computation time, leaving plenty of
room  for improvement in more complicated scenarios.

\section{The Weenie rule}\label{s:weenie}
A variant on Guts Poker is the addition of the ``Weenie rule'' \cite{S} penalizing overcautious play.
Under this rule, should all players drop, the player with highest hand- the ``weenie''- must match the
pot, thus doubling the stakes for the next round.  
It was shown in \cite{CCZ} that the optimal bloc strategy $p_1^*=1/2^{1/(n-2)}$
for standard guts is shifted under the Weenie rule to the more agressive strategy $p_1^*=1/3^{1/(n-1)}$,
as recorded in the following proposition.

\begin{proposition}[\cite{CCZ}]\label{wNE}
For $n$-player continuous Guts with the Weenie rule, the unique symmetric Nash equilibrium is
	$
	(1/3^{1/(n-1)}, \dots, 1/3^{1/(n-1)}),
	$
	and this is optimal for the bloc strategy game.
	Different from the standard case, it is a \emph{ strict} Nash equilibrium.
\end{proposition}

Here, we go further, investigating the general (nonbloc) case.
The surprising result is that in the case of the Weenie rule, quite differently than in the standard case,
the explicit Nash equilibrium strategy appears to be optimal not only against bloc strategies, but against
general nonbloc coalition strategies of players $2$-$n$.
That is, continuous Guts with the Weenie rule appears to be a rare example of a real-world multi-player
game with an exact solution: a simple pure strategy of threshold type.

\subsection{Payoff for $n=2,3$}\label{s:wpayoff}
We start by extending the following result of \cite{CCZ} to the case $n=3$.

\begin{proposition}[\cite{CCZ}]\label{2wpayprop}
For $2$-player continuous Guts with the Weenie rule, 
	\be\label{2walpha}
	\alpha(p_1^*,p_2^*)=
	\begin{cases}
		(1-2p_1^*-p_2^*)(p_1^*-p_2^*) & p_2^*\leq p_1^*,\\
		(1-2p_2^*-p_1^*)(p_1^*-p_2^*) & p_2^*> p_1^*.
	\end{cases}
	\ee
\end{proposition}

\begin{proposition}\label{wthm}
	For $3$-player continuous guts with the Weenie rule, $\beta$ is as for standard Guts, while
	the payoff for player $1$ is $\alpha(p_1^*,p_2^*,p_3^*)=$
	\ba\label{3walpha}
	\begin{cases}
		2p_1^*-p_2^*-p_3^*+(p_3^*)^3 -(p_1^*)^3 +3(p_2^*)^2p_3^*-3p_1^*p_2^*p_3^* , &
		p_1^*<p_2^*<p_3^*,\\
		2p_1^*-p_2^*-p_3^*+(p_2^*)^3 -(p_1^*)^3 +3(p_3^*)^2p_2^*-3p_1^*p_2^*p_3^* , &
		p_1^*<p_3^*<p_2^*,\\
		2p_1^*-p_2^*-p_3^*+(p_2^*)^3/2 + (p_3^*)^3 - 3 (p_1^*)^2p_2^*/2 
		- 3 (p_1^*)^2p_3^* + 3p_1^*p_2^*p_3^* , &
		p_2^*<p_1^*<p_3^*\\
		2p_1^*-p_3^*-p_2^*+(p_3^*)^3/2 
		+ (p_2^*)^3 - 3 (p_1^*)^2p_3^*/2 
		- 3 (p_1^*)^2p_2^* + 3p_1^*p_2^*p_3^* , &
		p_3^*<p_1^*<p_2^*\\
		2p_1^*-p_3^*-p_2^* - 2(p_1^*)^3 + (p_2^*)^3/2 +3 p_2^* (p_3^*)^2/2, &
		p_2^*<p_3^*<p_1^*,\\
		2p_1^*-p_3^*-p_3^* - 2(p_1^*)^3 + (p_3^*)^3/2 +3 p_3^* (p_2^*)^2/2, &
		p_3^*<p_2^*<p_1^*.\\
	\end{cases}
	\ea
\end{proposition}

\begin{proof}
	We first compute the correction to the standard guts payoff in the case that all players drop.

\textbf{Case 1}: $p_1^* < p_2^* < p_3^*$

Case i) $p_1, p_2, p_3 \leq p_1^*$, outcome zero.

Case ii) $p_1, p_3 \leq p_1^*; p_2 \in [p_1^*, p_2^*]$, outcome $1 \cdot (p_1^*)^2(p_2^* - p_1^*)$.

Case iii) $p_1, p_2 \leq p_1^*; p_3 \in [p_1^*, p_2^*]$, outcome $1 \cdot (p_1^*)^2(p_2^* - p_1^*)$

Case iv) $p_1 \leq p_1^*; p_2,p_3 \in [p_1^*, p_2^*]$, outcome $1 \cdot p_1^*(p_2^* - p_1^*)^2$

Case v) $p_1, p_2 \leq p_1^*; p_3 \in [p_2^*, p_3^*]$, outcome $1 \cdot (p_1^*)^2(p_3^*-p_2^*)$

Case vi) $p_1 \leq p_1^*; p_2 \in [p_1^*, p_2^*]; p_3 \in [p_2^*, p_3^*]$, outcome $1 \cdot p_1^*(p_2^* - p_1^*)(p_3^* - p_2^*)$

Total: $-(p_1^*)^3 + p_1^*p_2^*p_3^*$

\textbf{Case 2}: $p_2^* < p_1^* < p_3^*$

Case i), all fair play, outcome zero

Case ii) $p_2, p_1 \leq p_2^*; p_3 \in [p_2^*, p_1^*]$, outcome $1 \cdot (p_2^*)^2(p_1^* - p_2^*)$

Case iii) $p_2, p_1 \leq p_2^*; p_3 \in [p_1^*, p_3^*]$, outcome $1 \cdot (p_2^*)^2(p_3^* - p_1^*)$

Case iv) $p_2, p_3 \leq p_2^*; p_1 \in [p_2^*, p_1^*]$, outcome $-2 \cdot (p_2^*)^2(p_1^* - p_2^*)$

Case v) $p_2 \leq p_2^*; p_1, p_3 \in [p_2^*, p_1^*]$, outcome $-\frac{1}{2} \cdot p_2^*(p_1^* - p_2^*)^2$

Case vi) $p_2 \leq p_2^*; p_1 \in [p_2^*, p_1^*]; p_3 \in [p_1^*, p_3^*]$, outcome $1 \cdot p_2^*(p_1^*-p_2^*)(p_3^*-p_1^*)$

Total: $(p_2^*)^3/2 - \frac{3}{2}p_2^*(p_1^*)^2 + p_1^*p_2^*p_3^*$

\textbf{Case 3}: $p_2^* < p_3^* < p_1^*$.
Case i) fair play, outcome zero

Case ii) $p_2, p_3 \leq p_2^*; p_1 \in [p_2^*, p_3^*]$, outcome $-2 \cdot (p_2^*)^2(p_3^* - p_2^*)$

Case iii) $p_2, p_3 \leq p_2^*; p_1 \in [p_3^*, p_1^*]$, outcome $-2 \cdot (p_2^*)^2(p_1^* - p_3^*)$

Case iv) $p_2, p_1 \leq p_2^*; p_3 \in [p_2^*, p_3^*]$, outcome $1 \cdot (p_2^*)^2(p_3^* - p_2^*)$

Case v) $p_2 \leq p_2^*; p_1, p_3 \in [p_2^*, p_3^*]$, outcome $-\frac{1}{2}p_2^*(p_3^*-p_2^*)^2$

Case vi) $p_2 \leq p_2^*; p_3 \in [p_2^*, p_3^*]; p_1 \in[p_3^*, p_1^*]$, outcome $-2 \cdot p_2^*(p_3^* - p_2^*)(p_1^* - p_3^*)$

Total: $(p_2^*)^3/2 + \frac{3}{2}(p_2^*(p_3^*)^2) - 2p_1^*p_2^*p_3^*$

Summing with the standard payoffs given in \eqref{3alpha}, we obtain formulas (i), (iii), and (v)
of \eqref{3walpha}.  The remaining cases follow by invariance under permutation of $p_2^*$ and $p_3^*$.
\end{proof}

\subsection{Optimal strategy for $n=2,3$}\label{s:opt}

\begin{corollary}\label{weenieopt}
	For $2$- and $3$-player continuous Guts with the Weenie rule, the Nash equlibrium
	strategies $p_1^*=1/3$ and $p_1^*=1/\sqrt{3}$ are optimal against general (nonbloc) strategies,
	forcing $\underline{V}=0$.
\end{corollary}

\begin{proof}
	For $n=2$, this follows from Proposition \ref{wNE}, as all strategis are bloc strategies in this case.

	We focus therefore on the case $n=3$.

	\medskip
	{\bf Case I ($p_1^*\leq p_2^*\leq p_3^*$)} 
	\medskip

	Computing for $p_1^*=1/\sqrt{3}\leq p_2^*\leq p_3^*$ that
	$$
	\partial \alpha/\partial p_2^*= -1 + 6p_2^*p_3^* -3p_1^*p_3^*=
	(3p_2^*p_3^*-1) + 3p_3^*(p_2^*-p_1^*)\geq 0,
	$$
	we find for fixed $p_1^*=1/\sqrt{3}$, $p_3^*\geq 1/\sqrt{3}$, that
	$\alpha$ is minimized at the left endpoint $p_2^*=p_1^*=1/\sqrt{3}$,
	or $\alpha(1/\sqrt{3}, 1/\sqrt{3}, p_3^*)$.
	But, we know from the analysis of the bloc strategy case in \cite{CCZ} (here applied to
	player 1-2 bloc) that
	player 3 obtains return $\leq 0$, with equality if and only if $p_3^*=1/\sqrt{3}$.
	Thus, player 1 obtains a return $\geq 0$ in this case, with equality if and only if
	$p_1^*=p_2^*=p_3^*=1/\sqrt{3}$.

	\medskip

	{\bf Case II ($p_2^*\leq p_1^*\leq p_3^*$)} 
	Fixing $p_1^*=1/\sqrt{3}$, we obtain
	$$
	\tilde \alpha(p_2^*,p_3^*):= \alpha(1/\sqrt{3}, p_2^*,p_3^*)=
	2/\sqrt{3} -3p_2^*/2 -2p_3^* + (p_2^*)^3/2 + (p_3^*)^3  + \sqrt{3}p_2^*p_3^*.
	$$

	For fixed $p_3^*$, we find that
	\be\label{deriv}
	\partial \tilde \alpha/\partial p_2^*= 3(p_2^*)^2/2-3/2 + \sqrt{3} p_3^*.
	\ee
	This is nonnegative for 
	\be\label{nonneg}
	2/\sqrt{3} \leq p_3^*\leq 1,
	\ee
	hence is minimized in this case at $p_2^*=0$, or 
		\ba\label{minprob}
		\tilde \alpha(0,p_3^*)&= 2/\sqrt{3}  -p_3^* +  (p_3^*)^3 -p_3 \\
		&\geq  2/\sqrt{3} +   (p_3^*)^3 -2p_3. 
		\ea
	The righthand side of \eqref{minprob} is increasing for $2/\sqrt{3}\leq p_3^*\leq 1$, 
	and at endpoint $p_3^*=2/\sqrt{3}$ gives value 
	$2/\sqrt{3} + p_3^*((p_3^*)^2-2)= 2/\sqrt{3}(1+ 4/3 -2)>0$. 
	Thus, $\tilde \alpha(p_2^*,p_3^*)>0$ in case \eqref{nonneg}.

	In the remaining case,
	\be\label{indef}
	1/\sqrt{3} \leq p_3^*\leq 2/\sqrt{3}, 
	\ee
	noting that $\tilde \alpha(\cdot, p_3^*)$ is convex, we find that $\tilde \alpha$ is minimized
	at the stationary point $\partial \tilde \alpha/\partial p_2^*=0$, or
	\be\label{pdagger}
	p_2^*=p_2^\dagger(p_3^*):= \sqrt{1-2p_3/\sqrt{3}}.
	\ee
	Plugging \eqref{pdagger} into $\tilde \alpha$, we find that the minimum value of $\tilde \alpha$ is
	thus
	\ba\label{h}
	h(p_3^*)&=\tilde \alpha(p_2^\dagger(p_3^*), p_3^*)
	= 2/\sqrt{3} - 2(1-2p_3^*/\sqrt{3})^{1/2} - 2p_3^*\\
	&\quad
	+ (1-2p_3^*/\sqrt{3})^{3/2} + (p_3^*)^3 +\sqrt{3}(1-2 p_3^*/\sqrt{3})^{1/2}p_3^*.
  \ea

	We may readily compute that $p_2^\dagger(1/\sqrt{3})=1/\sqrt{3}$, hence
	\be\label{h0}
	h(1/\sqrt{3})=\alpha(1/\sqrt{3}, 1/\sqrt{3}, 1/\sqrt{3})=0,
	\qquad
	h(2/\sqrt{3})=\alpha(1/\sqrt{3}, 0 , 2/\sqrt{3})=\tilde \alpha(0, 2/\sqrt{3})>0
	\ee
	(the second already computed in case \eqref{nonneg}).
	Next, we compute
	$$
	h'(p_3^*)= (\sqrt{3}p_3^*/2-1)(1-2p_3^*/\sqrt{3})^{-1/2}
	+ (\sqrt{3} + 3/2)(1-2p_3^*/\sqrt{3})^{1/2}
	-2 + 3 (p_3^*)^2,
	$$
	and thus
	$$
	h''(p_3^*)=
	(1/\sqrt{3})
	\Big( C p_3^* - 1 - (p_3^*)^2 - \sqrt{3}\Big) (1-2p_3^*/\sqrt{3})^{-1/2},
	$$
	for $C= \sqrt{3}/2 + 2/\sqrt{3}= \sqrt{3}(1/2 + 2/3)> \sqrt{3} \geq 2p_3^*$.
	Observing for 
	$$
	g(p_3^*):= C p_3^* - 1 - (p_3^*)^2 - \sqrt{3}
	$$
	that
	$ g'(p_3)= C - 2p_3^* >0$, we find that $g$ is 
	maximized on $p_3^*\in [1/\sqrt{3}, \sqrt{3}/2]$ at the right endpoint
	$$
	g(2/\sqrt{3}=  2C /\sqrt{3} - 1 - 4/3 - \sqrt{3}= 14/6 - 6/6 -8/6 - \sqrt{3}<0,
	$$
	we find that $h''<0$ everywhere on $p_3^*\in [1/\sqrt{3}, \sqrt{3}/2]$.
	Thus, $h$ lies above its secant line between endpoints $p_2^*=1/\sqrt{3}$ and
	$p_2^*=2/\sqrt{3}$, and so $h(p_3)\geq 0$ by \eqref{h0}.
	Combining with case \eqref{nonneg}, we have that $\alpha\geq 0$ on the whole
	range $p_3^*\in [1/\sqrt{3},1]$,
	with equality if and only if $p_1^*=p_2^*=p_3^*=1/\sqrt{3}$.

	\medskip

	{\bf Case III ($ p_2^*\leq p_3^* \leq p_1^*$)} 
	\medskip

	Computing $\partial \alpha/\partial p_2^*)=
	-1 + 3(p_2^*)^2/2 + 3(p_3^*)^2/2 \leq -1 + 3(p_1^*)^2=0$, we find for fixed $p_1^*=1/\sqrt{3}$
	and $p_3^*\leq 1/\sqrt{3}$ that $\alpha$ is minimized on $[0, p_3^*]$ at $p_2^*=p_3^*$.
	But, by the results of \cite{CCZ} for the bloc case $\alpha(1/\sqrt{3},p_2^*,p_2^*)$ is $\geq 0$
	with equality if and only if $p_2^*=1/\sqrt{3}$.
	Thus, in this case, again, $\alpha\geq 0$ with equality if and only if $p_1^*=p_2^*=p_3^*=1/\sqrt{3}$.
\end{proof}

The result of Corollary \ref{weenieopt} for continuous Guts with the Weenie rule
is in striking contrast to the result for standard continuous Guts, giving an optimal winning strategy
for player 1 against a coalition of players 2-n.
This is illustrated in Figure \ref{fig_weenie} by a plot of the surface $\alpha(1/\sqrt{3}, p_2^*,p_3^*)$.

  \begin{figure}
		  \includegraphics[scale=0.65]{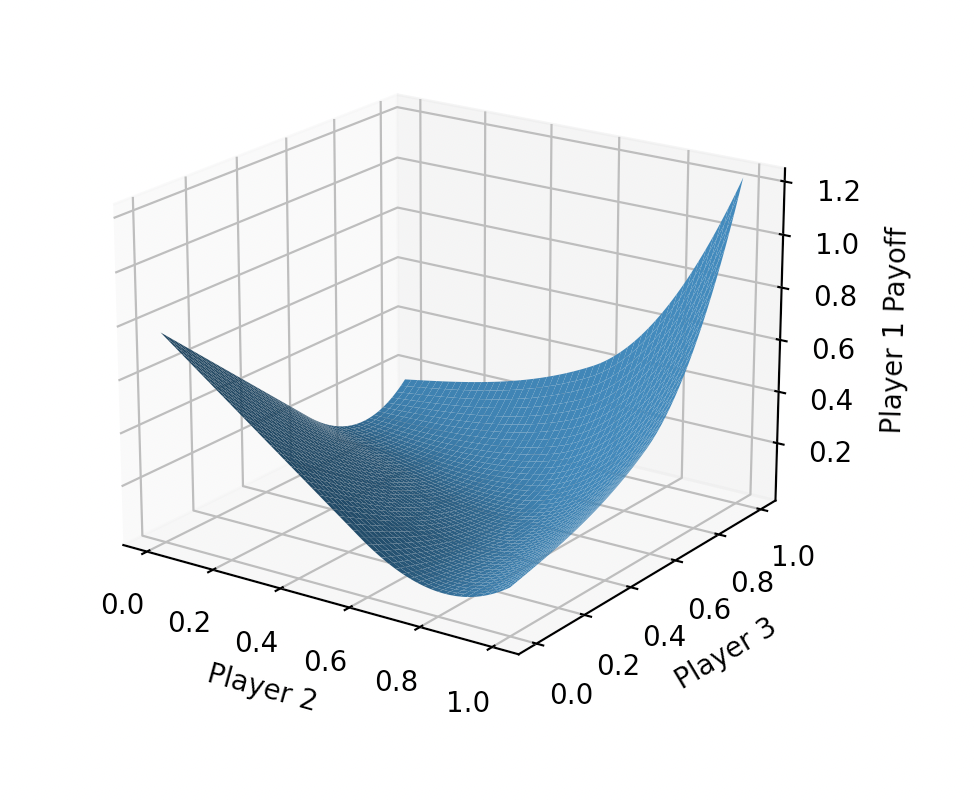}
	  \caption{Plot of $\alpha(1/\sqrt{3}, p_2^*,p_3^*)$ vs. $(p_2^*,p_3^*)$.}
	\label{fig_weenie}
  \end{figure}

\subsection{The case $n\geq 4$}\label{s:w4}
For the $n$-player game with $n\geq 4$, we investigated numerically, by 
checking for discretization $M=1,001$ the optimality property
$$
\min_{p_2^*, \dots, p_n^*}\alpha(1/3^{1/(n-1)},p_2^*, \dots, p_n^*)\geq 0.
$$
For $n=3,4,5$ all yielding (numerical) optimality, indicating optimality of the Nash equlibrium
strategy against general nonbloc strategies.
We conjecture that this property holds for general $n\geq 2$, in striking contrast to the results for
standard continuous Guts.

\subsection{Conclusions}\label{s:conclusions}
We see that $3$-player continuous Guts with the Weenie rule is a rare instance of a multi-player game with
exact minimax solution, or {\it strong Nash equilibrium}, namely, the symmetric equlibrium 
$(p_1^*,p_2^*,p_3^*)=(1/\sqrt{3}, 1/\sqrt{3}, 1/\sqrt{3})$. 
Indeed, it is the still rarer instance of a multi-player game with saddlepoint solution consisting
of pure rather than blended strategies, with pure strategies being of simple threshold type.

Note that this does {\it not} follow by concave-convexity of the payoff function as in the bloc case \cite{CCZ},
but by subtle dynamics of the model. 
Indeed, for symmetric multi-player ($n\geq 3$) games, considered as a $2$-player game between player $1$
and a player $2$-$n$ coalition, the associated payoff function $\alpha$  can never be concave (in $p_1$)-convex 
(in $(p_2,\dots,p_n$), by symmetry, unless $\alpha$ is linear.
The fact that standard Guts does not share this property is further evidence of its subtlety.

\section{Convergence of Fictition play}\label{s:convFP}
We include here a sketch of the proof of convergence of Fictitious Play for zero-sum $2$-player games given by
Jean Robinson \cite{R}, the first woman elected
to the National Academy of Sciences, as the solution of a RAND corporation prize problem posed by G. Brown.
The argument is based on the two ingredients of {\it duality} and {\it Hamiltonian structure}.

Let $P^n$ and $Q^n$ denote the vector of frequencies of strategies played respectively by players one and two,
so that $P^n/n$ and $Q^n/n$ are the associated empirical probability distributions.
Define $BR_1(q)$ to be the {\it best response} of player one against player two strategy $q$, i.e.
$argmax \, \phi(\cdot, q)$.
Likewise, define $BR_2(p)$ to be the {\it best response} of player two against player one strategy $p$, i.e.
$argmax  \,\phi(p, \cdot)$.

Fictitious play (FP) is thus given by the difference system
	\ba\label{FP}
	 P^{n+1}-P^n=  BR_1(Q^n/n) ,\qquad Q^{n+1}-Q^n= BR_2(P^n/n). 
	\ea
	For simplicity, we consider instead Continuous Fictitious Play (cFP)
	\ba\label{cFP}
	\dot P(t)=  BR_1(Q(t)/t) ,\qquad \dot Q(t)= BR_2(P(t)/t). 
	\ea
	for $t\geq 1$ (substituting $t$ for $n$ and $\dot P$ for difference quotients).

\medskip

{\bf Duality.} Both the value $V$ of the game and $\phi(P(t)/t,Q(t)/t)$ lie between 
$\max_p \phi(p,Q(t)/t)$ and $\min_q \phi(P(t)/t,q)$.  Thus, the {\it duality gap} 
$$
g:= \max_p \phi(p,Q(t)/t) -\min_q \phi(P(t)/t,q)
$$
converges to zero as $t\to \infty$ if and only if $\phi(P(t)/t, Q(t)/t)\to V$.

	\medskip

{\bf Hamiltonian structure}: Where $BR_j(\cdot)$ are smooth, 
$\mathcal{G}:= \phi(BR(Q/t),Q)- \phi(P,BR(P/t))= t g$ is
stationary with respect to the BR variables, and otherwise {\it decreasing}.
Thus, applying the chain rule, we find computing variations of $\mathcal{G}$ that
	\be\label{Ham}
	\hbox{\rm 
	$J \dot U {\geq} \nabla \mathcal{G}, \qquad U=\bp P\\Q\ep, 
	\qquad
	J:= \bp 0 & -A\\ A^T& 0\\ \ep $ skew, }
	\ee
	where the inequality $\geq$ in the first equation is understood to hold in the sense
	of subdifferentials, i.e., with respect to perturbations in the arguments of $\mathcal{G}$.
	That is, (cFP) has the form of a {\it Generalized Hamiltonian system}.
	
	\medskip
	{\bf Convergence for $2$-player games.}
Computing, we thus have $ \dot{\mathcal{G}} {\leq} (J \dot U) \cdot \dot U =0,$ and therefore 
$\mathcal{G}\leq \const$ and $g=\mathcal{G}/t \leq \const/t \to 0$, establishing convergence of (cFP).

	\subsection{The $n$-player case}
	For the $n$-player zero-sum case, defining $p_1(t), \dots p_n(t)$ to
	be the empirical probability distributions for players $1$-$n$,
	we point out that one may similarly define a duality gap
	\ba\label{ngap}
g&:= \max_{z}\phi_1(z, p_2(t), \dots, p_n(t)) + \max_{z}\phi_2(p_1(t), z , p_
(t), \dots, p_n(t)) + \dots \\
 &\quad + \max_{z}\phi_n(p_1(t), \dots, p_{n-1}, z ),
	\ea
where $\phi_j$ denotes the payoff function for player $j$, satisfying
$
g\geq \sum_{j=1}^n  \phi_j(p_1(t), \dots, p_n(t)) =0
$
with equality if and only if $(p_1(t), \dots, p_n(t))$ is a Nash equilibrium.
However, a computation for the $n$-player (cFP) analogous to \eqref{Ham} does not appear to yield anything useful,
as there is no cancellation in general for this case.

Recall from the introduction, however, that for continuous $3$-player (standard) Guts, Fictitious play
{\it does} converge, to the unique symmetric Nash equilibrium
$(p_1^*,p_2^*,p_3^*)=(1/\sqrt{2}, 1/\sqrt{2},1/\sqrt{2})$.
We plot the associated evolution of the gap $g(n)$ in Figure \ref{fig_FP}(i)
For a discretization of $501$ mesh points; convergence to Nash equilibrium $p_*=1/\sqrt{2}$
of the $n$th plays $p_j(n)$ for players $j=1,2,3$ (implying convergence of the associated empirical
probability distributions to a delta-function centered at $1/\sqrt{2}$) 
is illustrated in Figure \ref{fig_FP}(ii).
Note the eventual monotone decrease of both the gap $g(n)$ and the rescaled gap $G(n):=ng(n)$, 
similarly as in the $2$-player case.\footnote{Monotone decrease for (cFP) implies eventual monotone decrease
for (FP), as may be seen by the effective step size of $1/(n+1)$ for (FP) considered as a discretization
of (cFP).}
This may give some insight toward a proof of convergence in this case.
Namely, though the straightforward calculation showing decrease of $G$ in the $2$-player case 
breaks down in the general $3$-player case, the specifics of $3$-player Guts, accounted more carefully,
may yield the same result by a more circuitous route.

We see also that $G\sim \const$, suggesting a rate of decay $g(n)\sim 1/n$ similarly as in the $2$-player case.
We conjecture, but have not made a corresponding study,
that Fictitious play converges for $m$-player guts for all $m\geq 2$, with rate $g(n)\sim 1/n$.

  \begin{figure}
        \centering
        \begin{subfigure}[b]{0.45\textwidth}
		(a) \includegraphics[scale=0.25]{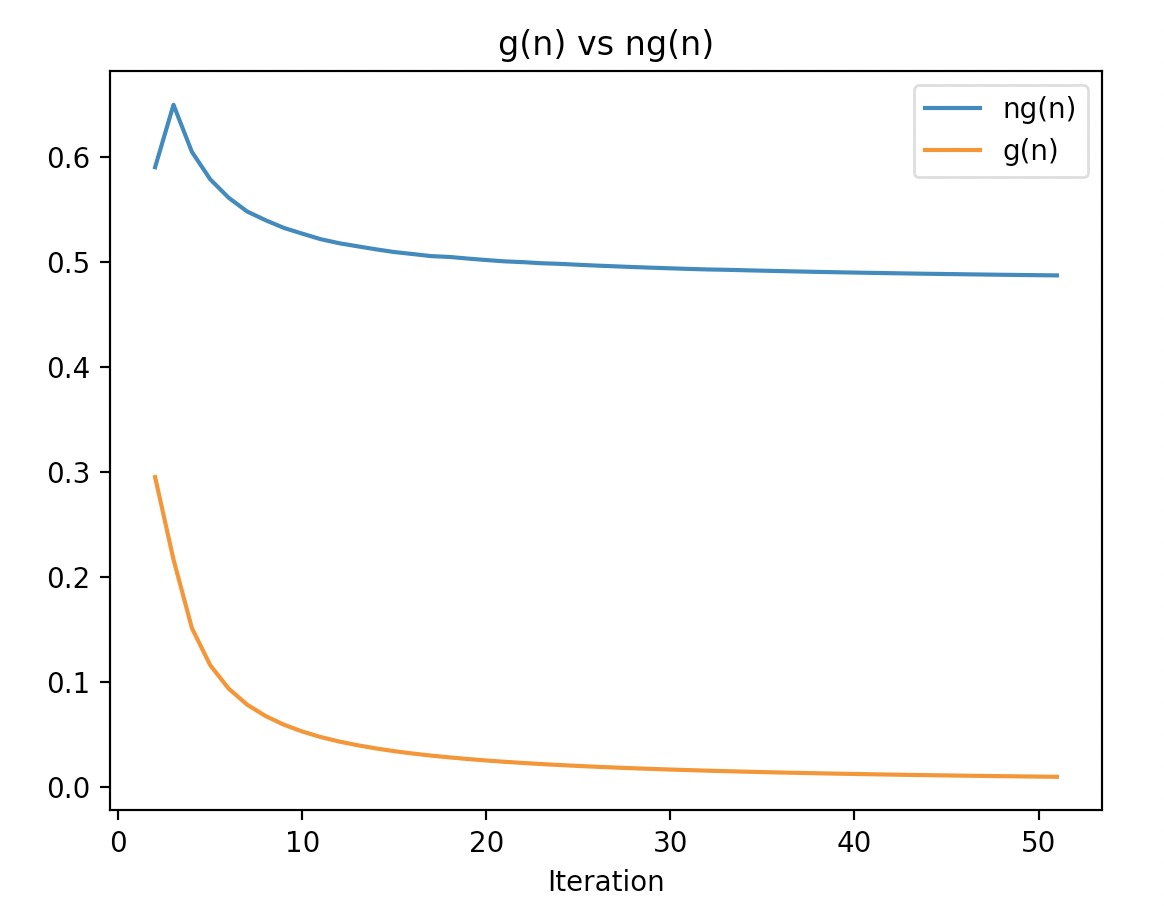}
		 \end{subfigure} 
		 \quad
		 \begin{subfigure}[b]{0.45\textwidth}
			 (b) \includegraphics[scale=0.25]{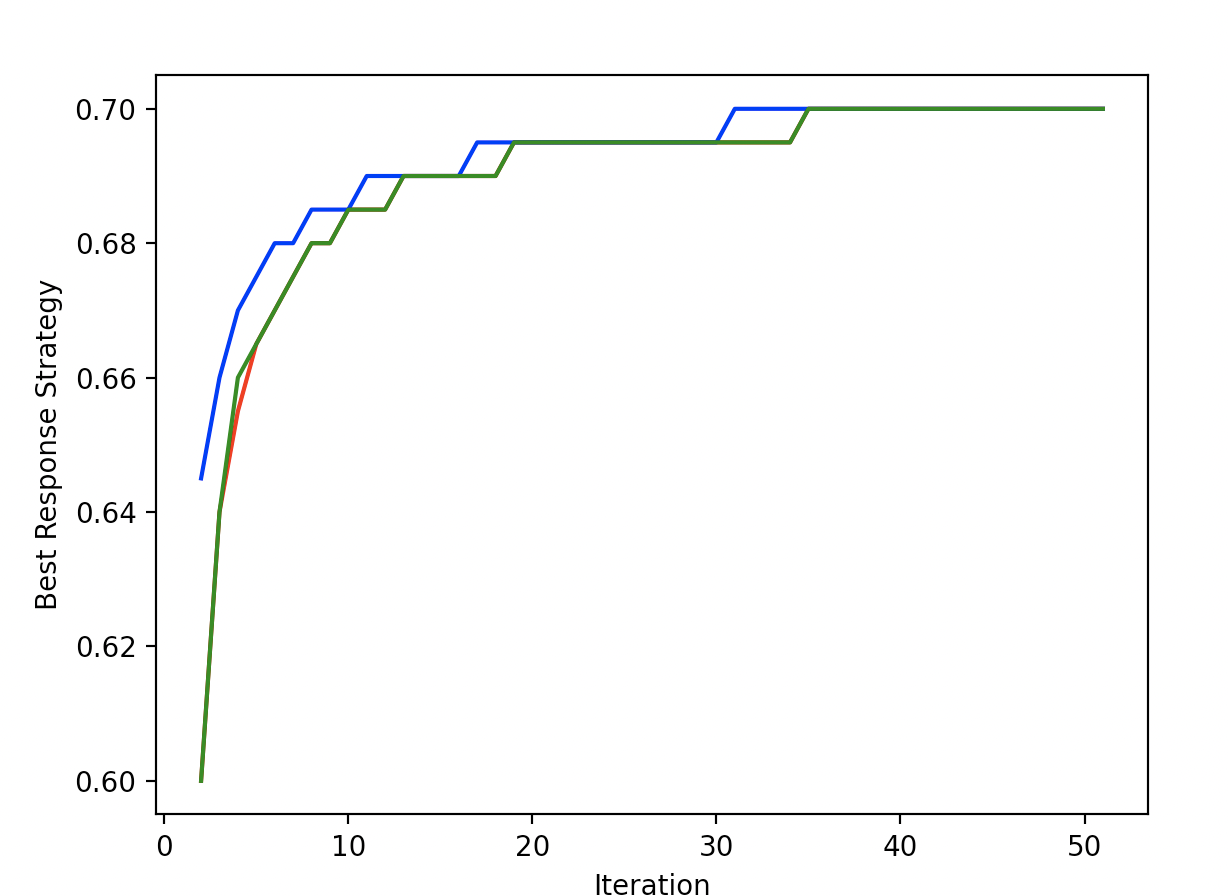}  
		 \end{subfigure}
		 	\caption{
				Graphs plotting for $n$th step of Fictitious play of $3$-player Guts
				a) gap $g(n)$ and rescaled gap $G(n):=ng(n)$ vs. $n$. b) plays $p_1^*(n)$, $p_2^*(n)$, and $p_3^*(n)$ vs. $n$. 
				}
	\label{fig_FP}
  \end{figure}

\subsection{The Jacob game}\label{s:jacob}
Using the 3 player fictitious play algorithm, we also study numerically a simple 3-player symmetric game, 
referred to here as the Jacob Game. 
Each player chooses a number: 1 or 2. If two players choose the same number, and the remaining player chooses 
the other, the players choosing the same number each pay a value equal to their number to the player who 
chose the other number. 
We can define a player's strategy $s_n$ to be the proportion of the time they choose 1. 
We can then characterize any equilibrium as the triplet $(s_1,s_2,s_3)$ of the $3$ player's strategies. 
The Jacob Game has 3 distinct Nash equilibria, of the form $(0,1,1)$, $(2/3,2/3,2/3)$ and 
$(1, 1/3, 1/3)$. 
Since the Jacob Game is symmetric, any permutation of the strategies in any one equilibrium is itself an equilibrium.

This game is particularly interesting, as it features multiple equilibria in which one player wins 
while the others lose in a symmetric game, additionally the only pure and only non-weak equilibrium is of this type. 
The fictitious play algorithm does converge for this example, converging to an equilibrium of the form $(0,1,1)$
(or a permutation thereof), in which player 1 wins $+2$ and players 2 and 3 each lose $-1$:
essentially the opposite of a coalition strategy for players 2-3.

\subsection{Convergence to coalition strategies: the Jacob game II}\label{s:coalition}
One may ask whether or when $n$-player Fictitious Play might converge to a coalition-type strategy instead of
a symmetric Nash equilibrium, in particular why it does not do so for $3$-player Guts.
As to the latter question, we note first that any limit point must be {\it some Nash equilibrium}, even if not
a symmetric one.  The winning bloc strategy for players $2$-$3$ on the other hand is not a Nash equilibrium,
since examination shows that one of players $2$, $3$ can always benefit by departing from the bloc strategy.
Second, we note that limiting strategies must necessarily be of {\it asynchronous} type, meaning that
all players independently choose randomly from their individual blended strategies.
As the winning strategy found above is, rather, of {\it synchronous} type, meaning that player 2's and player 3's
strategies are not chosen independently, it is disqualified for this second reason as well.

These considerations suggest the question whether there
exists a symmetric $3$-player game with an asymmetric Nash equilibrium consisting of a winning coalition strategy
for two players vs. the third, and if so what is its behavior under Fictitious Play.  
An interesting example in this regard is a symmetric, zero-sum, 3-player game
that we refer to as the Jacob game II. In this game, each player chooses from strategies $(1,2,3)$, with 
payoff schedule 
\be\label{mjacob}
1: \, \bp 0&0&0\\0&0&-2\\0&-2&0\ep; \quad
2 \, \bp 0&0&1\\0&0&-4\\1&-4& 16\ep; \quad
3: \, \bp 0&1&0\\1&8&-8\\0&-8&0\ep
\ee
displayed in the form $s_j: A_j$, where the payoff $\alpha(s_1,s_2,s_3)$ 
to player 1 for pure strategy choices $s_1, s_2, s_3=1,2,3$ is the $s_2$-$s_3$ entry of the matrix $A_j$.
Payoffs to players 2 and 3 are then determined by symmetry, as is possible in a consistent way thanks
to symmetry of the matrices $A_j$.
One may check by hand the zero sum property
$ \alpha(s_1,s_2,s_3)+ \alpha(s_2,s_1,s_3)+ \alpha(s_3,s_1,s_2)=0 $
for symmetric games.

This game possesses a pure nonstrict symmetric Nash equilibrium in which each player chooses strategy 1,
and also a blended nonstrict symmetric Nash equilibria in which each player chooses strategy $(s_1,s_2,s_3)$
with probabilities $(0,2/3,1/3)$.
It possesses also a pure asymmetric ``coalition-type'' strong Nash equilibria in which each player $j$ chooses 
$s_j=j$, with player 1 losing $-2$ and players 2 and 3 each gaining $+1$, along with its various permutations:
six in total. There may be other equilibria we have not found.
In experiments using Fictitious Play with randomly generated initial strategy ensembles for players 1-3,
we found that the process {\it always converged to a coalition-type equilibrium} of one type or the other,
in striking contrast to the situation for Guts poker.
(The standard Fictitious play protocol of choosing best response randomly in case of a tie explains the absence
of ``intermediate behavior'' of separatrix type.)

One may ask also whether there exist symmetric games with winning coalitions that are 
non-strong Nash equilibria, and what would be the result for Fictitious Play.  The expanded modification 
\ba\label{mjacobIII}
&1: \, \bp 0&0&0&2\\0&0&-2&2\\0&-2&0&2\\ 2 & 2 & 2 & 0\ep; \quad
2: \, \bp 0&0&1&2\\0&0&-8&2\\1&-8& 32&2\\ 2 & 2 & 2 & -4 \ep; \\
&3: \, \bp 0&1&0&2\\1&16  &-16 &2  \\0&-16  &0 &2\\ 2&2&2&2\ep ; \quad
4: \, \bp -4&-4&-4&0\\ -4&-4&-4&2\\ -4&-4&-4&-1\\ 0& 2&-1& 0 \ep
\ea
of \eqref{mjacob} features a winning coalition-type Nash equilibrium corresponding for players 1-3 to the
one for \eqref{mjacob}; however, it is no longer a strong equilibrium, due to the possibility of ``traitorous''
play in which one partner in the 2-player coalition colludes with the remaining player to improve their joint returns.
Experiment shows that Fictitious Play converges for this game also to this coalition-type equilibrium, 
or a symmetric permutation thereof.

The ``mega''-version
\ba\label{megajacob}
&1: \, \bp 0&0&0&2&2&2\\0&0&-2&2&2&2\\0&-2&0&2&2&2&\\ 
 2&2&2&-4&-4&-4 \\ 2&2&2&-4&-4&-4 \\ 2&2&2&-4&-4&-4 \ep; \quad
2: \, \bp 0&0&1&2&2&2\\0&0&-8&2&2&2\\1&-8& 32&2&2&2\\ 
2&2 & 2 &  -4&-4&-4 \\ 2&2 & 2 &  -4&-4&-4\\  2&2 & 2 &  -4&-4&-4 \ep; \\
&3: \, \bp 0&1&0&2&2&2\\1&16  &-16 &2&2&2  \\0&-16  &0 &2&2&2\\
2&2 & 2 &  -4&-4&-4 \\ 2&2 & 2 &  -4&-4&-4\\  2&2 & 2 &  -4&-4&-4 \ep ; \quad
4: \, \bp -4&-4&-4&2&2&2\\ -4&-4&-4&2&2&2\\ -4&-4&-4&2&2&2\\
2&2 & 2 &0&0&0   \\2&2 & 2 &0&0&-2 \\ 2&2 & 2 &0&-2&0 \ep; \\
&5: \, \bp -4&-4&-4&2&2&2\\ -4&-4&-4&2&2&2\\ -4&-4&-4&2&2&2\\
2&2 & 2 &0&0&1   \\2&2 & 2 &0&0&-8 \\ 2&2 & 2 &1&-8& 32 \ep; \quad
6: \, \bp -4&-4&-4&2&2&2\\ -4&-4&-4&2&2&2\\ -4&-4&-4&2&2&2\\
2&2 & 2 &0&1&0   \\2&2 & 2 &1&16&-16 \\ 2&2 & 2 &0&-16& 0 \ep; \quad
\ea
of \eqref{mjacobIII}, meanwhile, allows improvement not only through traitorous departure from coalition, but
also through jointly favorable modification of the coalition strategies.
Thus, though the described coalitions are optimal among synchronous strategies, they cannot be optimal 
with respect to ``synchronous'' coalition strategies where players strategies may (as for continuous guts)
be linked.  
Running fictitious play for this game yields convergence
to a coalition of either a player playing 2 and a player playing 3 against a player playing 1, or a player playing 5 and a player playing 6 against a player playing 4. This gives each member of the coalition +1, however if the players could play a synchronous strategy in which they synchronously decide to either play [2,3] or [5,6], they 
would on average get +1.5 each against the opposing player's best response. 
Thus, it appears that fictitious play can converge to an asynchronous coalition strategy, 
even when it is not optimal among the larger class of synchronous coalition strategies.
%

More generally, could there be examples
in which approximate coalitions spontaneously form and dissolve?  Or would there occur some more chaotic
possibility? This question seems worthy of further investigation.

\section{Sychronous vs. asynchronous coalition: two examples}\label{s:asynch}
We conclude with two examples demonstrating the possibility of a gap between the values forceable
by synchronous and by asynchronous coalitions, as mentioned in the introduction.

\subsection{Odd man in}\label{s:in}
Consider the following symmetric, zero-sum 3-player game.
Each player chooses a value 1, 2, or 3. If all choices are the same,
or all are different, there is no payoff. If two players choose a common number, however, and the third player
a different one, then the first two each pay a value of 1 to the third, i.e., the first two receive payoff -1
and the third +2. Clearly, the strategy distribution $(1/3, 1/3, 1/3)$ for player 1 gives average 
return of $+2/3$ if the other two players play the same number, and $-2/3$ if they play different numbers.
Thus, player 1 can force $\geq -2/3$.  On the other hand,  if players 2-3 choose with equal probability
between pairs of choices $(1,2)$, $(1,3)$, and $(2,3)$, then the average payoff to player 1 is independent
of player 1's choice of strategy, and equal to $(2/3)\times (-1) + (1/3)\times (0)=-2/3$.  Thus, players
2-3 can force a return of $\leq -2/3$ to player 1 by synchronous coalition play, and the value of the 
player 1 vs. players 2-3 game is $-2/3$. 

As described earlier, this is the maximum value that player 1 can force against 
either synchronous or asynchronous play; however, the value forceable by synchronous play of players 2-3
may in principle be larger. In fact, it is larger, as we now show.  

Let $y:=(y_1,y_2,y_2)$ and $z:=(z_1,z_2,z_3)$
denote probability distributions describing mixed strategies for players 2 and 3.
Then, it is readily computed that the payoff to player 1 is
\ba\label{aspay}
&\hbox{ \rm $\Psi_1(y,z):=2y\cdot v - (y_1+z_1)$ for player 1 choice 1,}\\
&\hbox{ \rm $\Psi_2(y,z):=2y\cdot v - (y_2+z_2)$ for player 1 choice 2,}\\
&\hbox{ \rm $\Psi_3(y,z):=2y\cdot v - (y_3+z_3)$ for player 1 choice 3.}\\
\ea
The value $\Psi(y,z):= \max_j \Psi_j(y,z)$ is thus the minimum value forceable by choice $(y,z)$,
and
$$
\overline{V}=\min_{y,z} \Psi(y,z)
$$
is the minimum value forceable by players 2-3 via asynchronous play, and by continuity of $\Psi$ is
achieved for some feasible pair of strategies $(y_*,z_*)$.

Noting that the average of $\Psi_j$ is 
$$
2 y\cdot z - (1/3)\sum_j (y_j+z_j)= 2 y\cdot z - (2/3)\geq -2/3,
$$
we have that $\Psi(y,z)\geq -2/3$, with equality if and only if simultaneously $y\cdot z=0$
and $(y_j+z_j)=2/3$ for all $j$. But, these together imply that one of each pair $y_j, z_j$ has
value zero and the other value $2/3$, which is impossible to reconcile with $\sum_j y_j=\sum_j z_j=1$.
Thus, evaluating at $(y,z)=(y_*,z_*)$, we obtain $\overline{V}= \Psi(y_*,z_*)>-2/3$, verifying
that there is indeed a gap between this value and the value $-2/3$ forceable by synchronous coalition play.
Indeed, the optimum asynchronous strategy 
can be shown to be $y_*=(1,0,0)$, $z_*=(0, 1/2, 1/2)$, forcing an expected
payoff to player 1 of $\leq -1/2$: thus, a gap of $-1/2-(-2/3)=1/6$.

To complete the picture, recall from \cite[Prop. A.2]{CCZ} that for a symmetric zero-sum game, the
symmetric Nash equlibrium is equal to the optimum strategy for player 1 in the modified ``bloc strategy''
game of player 1 vs. players 2-3, where the latter are required to play the same strategy,
in this case the game with payoff 
\be\label{blockpay}
\Psi(y,y)=\min_j 2(|y|^2 - y_j).
\ee
By Jensen's inequality, $|y|^2/3\geq [(\sum_j y_j)/3]^2= 1/9$, or $|y|^2\geq 1/3$, while $\min_j y_j
\leq (\sum_j y_j)/3$, both with equality if and only if $y_j\equiv 1/3$.
Thus, $\Psi(y,y)\geq 0$ with equality only if $y=(1/3,1/3, 1/3)$, and so $x,y,z=(1/3,1/3,1/3)$ is
the unique strict symmetric Nash equlibrium, returning value $0$, as it must for a symmetric zero-sum game.

\subsection{Odd man out}\label{s:in}
Next, consider the same game, but with payoff function multiplied by $-1$: that is, the ``reverse'' game,
in which the odd player is penalized instead of rewarded.
Again, the symmetric Nash equilibrium
is $x=y=z=(1/3,1/3,1/3)$, returning payoff zero to all players.  
An optimum synchronized strategy is a blend of pure strategy pairs 1-1, 2-2, 3-3, each chosen with probability 
$1/3$, yielding again value $(1/3)(-2)=-2/3$

However, the optimum value forceable by asynchronous coalition by players 2-3 is now
$
\tilde{V}= \min_{y,z}\tilde \Psi(y,z)$, where
$$
\tilde \Psi(y,z): =\min -\psi_j(y,z)= -2 y\cdot z + \max_j (y_j+z_j),
$$
and this, by Cauchy-Schwarz and Jensen's inequalities $y\cdot z\leq |y||z|$ 
and $|y|\leq (\sum_j y_j)/\sqrt{3}=1/\sqrt{3}$, together with the fact that maximum is greater than or equal
to average, is greater than or equal to $-2/3 + 2/3=0$, with equality if and only if
$y=z=(1/3,1/3,1/3)$, the Nash equilibrium values.
Thus, in the reverse game as in the forward one, player one can force at most $-2/3$, the value forceable
by players 2-3 by synchronized coalition play; hence, there is again a gap. 
However, players 2-3 can force at most $0$ by ansynchronous coalition play; that is, they cannot guarantee
by such strategy a joint profit for themselves.

\subsection{Discussion}\label{s:asdisc}
We note, as described in the introduction, that a similar situation appears to hold for (standard) 
3-player continuous guts. That is, the optimum synchronous coalition strategy appears to be strict up 
to symmetry and genuinely mixed-type in the sense that it has no pure-strategy representative,
so that the optimum asynchronous strategy must return a strictly larger maximum expected value to player 1
and there is a gap between values forceable by players 2-3 via synchronous vs. asynchronous coalition play.
A very interesting open question is whether for Guts, as in the forward example, 
there is also a gap between the value forceable
by asynchronous coalition and the symmetric Nash equlibrium return, so that 
players 2-3 can guarantee by asynchronous play a joint profit to themselves,
or whether as in the reverse example they can guarantee at most a zero outcome by asynchronous play.
We conjecture that for Guts, the latter holds;
to determine whether or when this is true for general games is an 
extremely interesting problem worthy of further investigation.

For games with a gap between values forceable by synchronous vs. asynchronous coalition values, 
as seem from the above discussion to be somewhat typical,
it is a further very interesting philosophical question what are the implications 
for player 1.  For, following the classical worst-case scenario analysis, player 1 cannot guarantee better 
than the value $V$ against synchronous play.  However, suppose that players 2 are prevented from communicating 
or even knowing which round is being played, so that their play is truly asynchronous. 
Then, they cannot in {\it their} worst-case scenario force less than the value $\overline V>V$ 
obtainable by asynchronous play.  But, is there any way for player 1 to take advantage of this gap, 
other than negotiation outside the game?  This seems to be a subtle side-issue for symmetric zero sum games,
apart from the standard one of determining value/optimal play when there exist successful asynchronous coalition 
strategies.

\end{document}